\documentclass[11pt]{amsart}

\usepackage{amssymb, url, mathrsfs, amsfonts, MnSymbol, amsmath}
\usepackage[all,cmtip]{xy}
\usepackage[pdftex,lmargin=1.25in, rmargin=1.25in,tmargin=1.25in,bmargin=1.25in]{geometry}
\usepackage{color}
\DeclareMathAlphabet{\mathscr}{OT1}{pzc}{m}{it} 

\usepackage{amsthm}
\usepackage{euscript}
\usepackage{dsfont}
\usepackage{hyperref}
\usepackage[all]{xy}
\usepackage{tikz}
\usepackage{mathtools}
\usepackage{cite}
\usepackage{stmaryrd}

\usetikzlibrary{matrix,arrows,decorations.pathmorphing,backgrounds,automata}
\usepackage{nicefrac}

\newcommand{\field}[1]{\mathbb{#1}}
\newcommand{\R}{\field{R}} 
\newcommand{\C}{\field{C}} 
\newcommand{\Q}{\field{Q}} 
\newcommand{\Z}{\field{Z}} 
\newcommand{\G}{\field{G}} 
\newcommand{\h}{\mathfrak{h}} 
\newcommand{\T}{\mathbb{T}} 

\newcommand{\inv}{^{-1}}
\newcommand{\Id}{\operatorname{Id}}

\newcommand{\End}{\operatorname{End}}

\newcommand{\Nm}{\operatorname{Nm}}
\newcommand{\GL}{\operatorname{GL}}
\newcommand{\SL}{\operatorname{SL}}

\newcommand{\Res}{\operatorname{Res}}
\newcommand{\pr}{\operatorname{pr}}

\theoremstyle{plain}

\newtheorem{thm}{Theorem}

\numberwithin{thm}{subsection}
\newtheorem{cor}[thm]{Corollary}

\newtheorem{lem}[thm]{Lemma}

\newtheorem{prop}[thm]{Proposition}

\theoremstyle{definition}

\newtheorem{defn}[thm]{Definition}
 
\theoremstyle{remark}

\newtheorem{rem}[thm]{Remark}

\title[Differential Operators for Hilbert Modular Forms]{Families of Differential Operators for Overconvergent Hilbert Modular Forms}

\author{Jon Aycock}
\date{\today}
\address{Jon Aycock\\
Department of Mathematics\\
University of Oregon\\
Fenton Hall\\
Eugene, OR 97403-1222\\
USA}
\email{jaycock@uoregon.edu}

\begin{document}

\begin{abstract}
We construct differential operators for families of overconvergent Hilbert modular forms by interpolating the Gauss--Manin connection on strict neighborhoods of the ordinary locus. This is related to work done by Harron and Xiao and by Andreatta and Iovita in the case of modular forms and by Zheng Liu for Siegel modular forms. It has applications to $p$-adic $L$-functions of CM fields.
\end{abstract}

\maketitle

\tableofcontents

\section{Introduction}
\subsection{Motivation}
In \cite{Serre}, Serre introduced the idea of using $p$-adic families of modular forms to $p$-adically interpolate values of $L$-functions. In particular, he used the family of Eisenstein series with $q$-expansion
\begin{equation}
2G_{2k}(q) = \zeta(1-2k) + 2\sum_{n=1}^\infty \sigma_{2k-1}(n)q^n.
\end{equation}

He showed that congruences between the coefficients of $q^n$ for $n \geq 1$ for particular values of $k$ imply the existence of such a congruence between the constant terms $\zeta(1-k)$ as well, and used this to show that (an appropriate normalization of) the Riemann zeta function is $p$-adically continuous when restricted to inputs of negative odd integers. This is one construction of the Kubota--Leopoldt $p$-adic zeta function, which interpolates the values of the Riemann zeta function at negative odd integers.

Serre was able to use a well-known family of modular forms whose values at the cusp $\infty$ are equal to special values of the Riemann zeta function. Though the result was generalized to Dedekind zeta functions for totally real fields in \cite{DeligneRibet}, these can be tricky to come up with in more generality. One important addition to the theory was Katz's use of differential operators, in particular those built from the Gauss--Manin connection $\nabla$, as a key ingredient in the construction of these families in \cite{KatzQI,Katz}. Katz's operators give a $p$-adic analog of the Maass--Shimura operators which Shimura used to prove algebraicity results in \cite{Shi76,Shi00}, and which were adapted by Harris to prove algebraicity results for higher rank groups in \cite{Har81,Har86}.

Katz leveraged this idea to $p$-adically interpolate the zeta function for a CM field $K$ using Hilbert modular forms attached to its maximal totally real subfield $K^+$. He started with a holomorphic Eisenstein series and used successive powers of these differential operators to produce a family of not necessarily holomorphic Eisenstein series. By Damerell's formula (in e.g. \cite{GSDamerell}), sums of the values of these Eisenstein series at the CM points of the modular curve give the central values $L(\chi,s_0)$ of the $L$-functions for a specific class of Hecke characters, known as characters of type $A_0$. Viewing the zeta function as a function on the space of characters, this fact reduces the $p$-adic interpolation of $\zeta_K$ to the study of how powers of the Gauss--Manin connection behave $p$-adically.

As noted in his introduction, Katz's method only succeeds when we can choose an ``ordinary" CM type, which is only possible when every prime above $p$ in the totally real field $K^+$ splits in $K/K^+$. This is due to the fact that his differential operators are only defined over the ordinary locus, which does not contain the CM points that are supersingular at $p$. The present work extends these operators to be defined on the overconvergent loci, using methods from \cite{L1} and using the geometry developed in \cite{AIP,AIP2}. This extension allows Damerell's formula to be used in more general situations, whenever the Eisenstein series is defined at the CM points.

\subsection{Description of Results}

The present work culminates in the following theorem, which is Theorem \ref{theorem} in the body of the paper:

\begin{thm}\label{introtheorem}
Fix a tuple $\underline{v} = (v_i)$ with each $v_i>0$. For each embedding $\sigma \colon F \to K$, and any $k \geq 1$, there is a differential operator $\nabla_\sigma^k$ acting on families of nearly $\underline{v}$-overconvergent Hilbert modular forms, which raises the weight by $2k\sigma$ and the type by $k$. The operators $\nabla_\sigma^k$ and $\nabla_\tau^\ell$ commute for any pair of embeddings $\sigma$ and $\tau$.
\end{thm}

These $\nabla_\sigma^k$ are the $p$-adic analogs of the Maass-Shimura operators in the Archimedean case, as we prove in Section \ref{NEARLYComparison}. The absence of these ``overconvergent" operators in \cite{Katz} prevented Katz from constructing $p$-adic $L$-functions for CM fields in the case that the CM points were not ordinary; i.e., when $p$ is not split in $K/K^+$. The overconvergent operators for elliptic modular forms have been previously constructed, and are used in \cite{AI} to construct $p$-adic $L$-functions for quadratic imaginary fields; our construction lays the groundwork to generalize that construction to a general CM field.

In defining the differential operators, we work rigid analytically. However, we give a discussion of the integrality of the operator in Section \ref{IntegralitySection}.

As in \cite{AIP2}, we are careful to make a distinction between two possible meanings of the phrase ``Hilbert modular form." Given a totally real field $F$, the term may refer to the automorphic forms on either of the following groups:
\begin{equation} G = \operatorname{Res}_{\mathcal{O}_F/\Z} \GL_2, \quad \text{or} \quad G^* = G \times_{\operatorname{Res}_{\mathcal{O}_F/\Z}\G_m} \G_m. \end{equation}
For any commutative ring $R$, the $R$-points of $G$ are simply the $2 \times 2$ invertible matrices with entries in $\mathcal{O}_F \otimes_\Z R$, while the $R$-points of $G^*$ are those matrices from $G(R)$ with determinant in $R^\times \subset (\mathcal{O}_F \otimes_\Z R)^\times$. Each of these groups has an advantage: $G$ has a nicer Hecke theory including a commutative Hecke algebra; while $G^*$ allows us to use geometric tools, since its associated PEL-type moduli problem is representable by a Shimura variety. The inclusion $G^* \subset G$ gives a restriction map from the space of automorphic forms on $G$ to those on $G^*$, and in fact there is an explicit geometric criterion that picks out the space of automorphic forms for $G$ inside the space of automorphic forms for $G^*$. In the following, we focus first on $G^*$ so that we may use the geometric tools it affords us, after which we shift our attention to this criterion which allows us to transport our results to the group $G$. We get the following theorem, Theorem \ref{secondtheorem} in the body of this paper, which allows these operators to be used in either situation:

\begin{thm}
The differential operators $\nabla_\sigma^k$ constructed in Theorem \ref{introtheorem} preserve the space of Hilbert modular forms for $G$ inside the space of Hilbert modular forms for $G^*$.
\end{thm}

In Section 2, we present the geometric construction of holomorphic Hilbert modular forms, followed by families of overconvergent Hilbert modular forms, on the group $G^*$ from \cite{AIP2}. In Section 3, we generalize to nearly holomorphic and nearly overconvergent forms, including the scaffolding of $(\mathfrak{g},Q)$-modules which are used to construct the differential operator as in \cite{L1}. In Section 4, we transport the results of the previous section to the group $G$ using a criterion similar to the one mentioned above.

\subsection{Relationship to Recent Developments}

This paper adds to the recent investigations into extending $p$-adic differential operators past the ordinary locus, or into when the ordinary locus is empty, including those in \cite{dSG16, dSG19,EM20,U,L1,L2}. Shortly after posting this paper on the arXiv, the author learned from Giacomo Graziani that his dissertation \cite{Graziani} (currently in preparation) is closely related. His paper uses the VBMS (vector bundles with marked sections) formalism present in e.g. \cite{AI}, while we use $(\mathfrak{g},Q)$-modules as in \cite{L1}.

In \cite{HX}, the authors ask whether or not their construction can be adapted to avoid a choice of the Hodge filtration. Building on the work of \cite{AI,L1}, this work answers that question in the affirmative. Additionally, other works note that in order to use nearly overconvergent Hilbert modular forms to interpolate the values of $p$-adic $L$-functions, a ``canonical splitting of the Hodge filtration" must be chosen; see for example the introduction to \cite{AI} for this in the situation of Damerell's formula. Our methods translate this into a need for a canonical choice of frame for the relative de Rham cohomology of the modular curve in the nearly holomorphic case. Seeing as a canonical trivialization (i.e., a frame) for the modular sheaf $\underline{\omega}$ must already be chosen in Serre's setting, this is a natural bit of data to consider. For the nearly overconvergent case, we take preimages of the Andreatta--Iovita--Pilloni torsors for overconvergent Hilbert modular forms in the space of frames for the de Rham cohomology.

This construction works for elliptic modular forms, and seems eminently generalizable to be used for the automorphic forms on more groups, such as Hilbert--Siegel and hermitian modular forms. The author of the present paper is currently working on a sequel extending his approach to the higher rank setting, in particular Hilbert--Siegel modular forms.

\subsection{Acknowledgements}
I am grateful to have received support from NSF Grant DMS-1751281. I would like to thank my advisor Ellen Eischen, as well as Nicolas Addington, Sean Howe, and Zheng Liu. In addition, I would like to thank Giacomo Graziani for alerting me to the fact that his dissertation is about the same subject; I hope that our two perspectives give a fuller view than either would have alone.

\section{Hilbert Modular Forms}
Fix a totally real field $F$ of degree $[F:\Q] = d>1$. In this section, we will build up the geometric theory of holomorphic and overconvergent Hilbert modular forms of level $\Gamma_1(N)$. Here a Hilbert modular form is an automorphic form on the group $G^*$.

\subsection{Weights}\label{starweights}

Let $K$ be a local field that splits $F$, so that $F \otimes_\Q K = K^d$; it should often be a $p$-adic field for some chosen prime $p$, but in some contexts it is allowed to be $\R$. We consider the algebraic group $\T = \Res_{\mathcal{O}_F/\Z}\G_m$. The weight space of classical Hilbert modular forms is the space of characters of the split torus $\T_{/K}$.

Each embedding $\sigma \colon F \to K$ determines a character of $\T_{/K}$. Let $I$ denote the set of such embeddings. All algebraic weights (i.e., classical weights with no nebentypus at $p$, if $K$ is a $p$-adic field) are of the form $\sum_{\sigma} k_\sigma \sigma \in \Z[I]$, often recorded as a tuple $(k_\sigma)_\sigma$.

We fix a prime $p$, and a finite extension $K$ of $\Q_p$. We define the {weight space} $\mathcal{W}$, a rigid analytic space defined over $K$ associated to the algebra $\Z_p\llbracket\T(\Z_p)\rrbracket$. The $\C_p$-points of $\mathcal{W}$ parametrize continuous homomorphisms $\T(\Z_p) \to \C_p^\times$, which are the $p$-adic weights of Hilbert modular forms. As a rigid analytic space, it is isomorphic to a finite disjoint union of open unit balls of dimension $d$, where each component is labeled by a (finite order) character of the torsion subgroup of $\T(\Z_p)$. A character of $\T(\Z_p)$ is called algebraic if it is in $\Z[I]$, and locally algebraic if it is the product of a finite order character and an algebraic character. Locally algebraic characters are dense in the weight space.

For any value\footnote{If $p$ ramifies with index $e$, $w \in \frac{1}{e}\Z$. Writing $(p)^w$ should make sense as an ideal of $\mathcal{O}_F$, and we will sometimes write $p^w$ abusively even if there is no such element.} $w \in v(\mathcal{O}_F)$, the rigid analytic group $\T$ has two important subgroups. The first, $\T_w^0$, consists of units congruent to $1$ modulo $(p)^w$. The second, $\T_w$, is generated by $\T_w^0$ and $\T(\Z_p)$.

Let $A$ be an affinoid algebra, and let $\kappa \colon \T(\Z_p) \to A^\times$ be a family of characters parametrized by the affinoid space $\mathcal{U} = \operatorname{Sp}A$. As asserted in \cite{AIP2}, such a character is $w$-analytic for some $w$, meaning that it factors as a composition as follows for some linear map $\psi$:
\begin{equation}
	\T_w^0(\Z_p) = 1 + p^w(\mathcal{O}_F \otimes \Z_p) \xrightarrow{\log_F} p^w\mathcal{O}_F \otimes \Z_p \xrightarrow{\psi} pA \xrightarrow{\exp} 1 + pA \subset A^\times.
\end{equation}

There is a universal character $\kappa^{un} \colon \mathcal{W} \times \T(\Z_p) \to \C_p^\times$, where $\kappa^{un}(x,t)$ evaluates the character associated to the point $x$ at the input $t \in \T(\Z_p)$, corresponding to the natural character $\T(\Z_p) \to \mathcal{O}_K\llbracket\T(\Z_p)\rrbracket$ sending $g$ to $g$. The universal character is $1$-analytic. Any family of characters as above may be specified by pulling back the universal character by a map $\mathcal{U} \to \mathcal{W}$.

\subsection{Moduli Setup}
Let $\mathfrak{c}$ be an ideal of $F$, $p$ a prime number, and $N \geq 5$ an integer prime to $p$. A $\mathfrak{c}$-polarized Hilbert--Blumenthal abelian variety (HBAV, we suppress $\mathfrak{c}$ from the notation) defined over a base scheme $S$ is a string $(A,\iota,\psi,\lambda)$ where
\begin{itemize}
	\item $A \to S$ is an semi-Abelian scheme\footnote{A group scheme which is an extension of an Abelian variety by a torus} of relative dimension $d$,
	\item $\iota \colon \mathcal{O}_F \to \End(A)$ is a real multiplication on $A$,
	\item $\psi \colon \mu_N \otimes_\Z \mathcal{D}_F\inv \to A$ is a closed embedding, and
	\item $\lambda \colon A \otimes_{\mathcal{O}_F} \mathfrak{c} \to A^\vee$ is a $\mathfrak{c}$-polarization.
\end{itemize}

If $\mathfrak{c}$ is principal, we say $(A,\iota,\psi,\lambda)$ is principally polarized. We let $\mathfrak{c}$ and the degree of the isogeny $\lambda$ be prime to $p$. Often, we will use $A$ as a shorthand for the entire string.

Let $S$ be a scheme over $\operatorname{Spec}\Z\left[N\inv\right]$, $X \to S$ be the moduli space of HBAVs over $S$, and $\pi \colon \mathcal{A} \to X$ be the universal HBAV over $X$. The modular sheaf is $\underline{\omega} = \pi_*\Omega^1_{\mathcal{A}/X}$. This is locally free of rank $d$ as a sheaf of $\mathcal{O}_X$ modules, with an action of $\mathcal{O}_F$. There is a largest open subscheme $X^R$ of $X$, such that $\underline{\omega}$ restricts to a locally free sheaf of $\mathcal{O}_F \otimes_\Z \mathcal{O}_{X^R}$-modules. If we use $S = \operatorname{Spec}\mathcal{O}_K$ as our base, the complement of $X^R$ in $X$ is empty if $p$ does not divide the discriminant of $F$. In general, it is of codimension $2$ in the fiber of $X$ over the closed point of $S$. The space $X^R$ is known as the Rapoport locus, and the condition that a HBAV $A$ corresponds to a point of $X^R$ is called the Rapoport condition.

\begin{rem}
In the definition of HBAV, we really want $A$ to be an Abelian variety, rather than a semi-Abelian scheme. If we take this alternate definition, and write $Y$ as the moduli space of these varieties, then $X$ is a projective toroidal compactification of $Y$. We write $C = X \setminus Y$, and refer to $C$ as the boundary or the cusps. Allowing semi-Abelian schemes in the definition of HBAV is similar to talking about the moduli space of generalized elliptic curves and the universal generalized elliptic curve when discussing elliptic modular forms.
\end{rem}

\subsection{Holomorphic Hilbert Modular Forms}

Over $X_{/K}$, the bundle $\underline{\omega}$ decomposes as a direct sum $\underline{\omega} = \bigoplus_\sigma \underline{\omega}_\sigma$, so that $c \in \mathcal{O}_F$ acts on $\underline{\omega}_\sigma$ as $\sigma(c) \in K$. For a tuple of positive integers $\kappa = (k_\sigma)_\sigma \in \Z_{\geq 1}[I]$, we can form the line bundle $\underline{\omega}_\kappa \coloneqq \underline{\omega}_{\sigma_1}^{\otimes k_{\sigma_1}} \otimes \dots \otimes \underline{\omega}_{\sigma_d}^{\otimes k_{\sigma_d}}$ over $X_{/K}$.

\begin{defn}
A Hilbert modular form of level $\Gamma_1(N)$ and weight $\kappa = (k_\sigma)_\sigma$ is a global section of $\underline{\omega}_\kappa$. The $K$-vector space of Hilbert modular forms of level $\Gamma_1(N)$ and weight $\kappa$ is thus $H^0(X,\underline{\omega}_\kappa)$.
\end{defn}

We remark that $T_{\underline{\omega}}^\times \coloneqq \operatorname{Isom}_{X,F}(\mathcal{O}_F \otimes_\Z \mathcal{O}_X,\underline{\omega})$, the $\mathcal{O}_F$-frame bundle of $\underline{\omega}$ is a $\T$-torsor.\footnote{i.e., the frame bundle of $\underline{\omega}$ taking the $\mathcal{O}_F$-module structure into account. The subscripts $X$ and $F$ denote the fact that these isomorphisms should be $\mathcal{O}_X$-linear and $\mathcal{O}_F$-linear, respectively.} This means that the ring of functions of $T_{\underline{\omega},/K}^\times$ is graded by the characters of $\T_{/K}$, or rather by the algebraic weights of Hilbert modular forms. The graded portion corresponding to the weight $\kappa = (k_\sigma)_\sigma$ is $\underline{\omega}_\kappa$, and we find that the ring of Hilbert modular forms is the space of global sections of $\bigoplus_{\kappa \in \Z[I]}\underline{\omega}_\kappa = \mathcal{O}_{T_{\underline{\omega}}^\times}$.

Each line bundle $\underline{\omega}_\kappa$ can be recovered from $\mathcal{O}_{T_{\underline{\omega}}^\times}$ by considering the homogeneous functions with the property that for any $g \in \T$, any HBAV $A$, and any frame $\alpha \colon \mathcal{O}_F \otimes_\Z \mathcal{O}_X \xrightarrow{\sim} \underline{\omega}$,
\begin{equation}\label{explicithomogeneitykappa}
f(A,\alpha g) = \left(\prod \sigma(g)^{-k_\sigma}\right) f(A,\alpha) = \kappa(g\inv) f(A,\alpha).
\end{equation}
Somewhat less explicitly, we may view these as functions from $T^\times_{\underline{\omega}}$ to the representation $W_\kappa$ of $\T$, defined over $K$. The action of $\T$ is $g \cdot w = \kappa(g)w$, and the functions are homogeneous in the sense that $f(A,\alpha g) = g\inv \cdot f(A,\alpha)$.

\begin{rem}
Homogeneous functions on $T^\times_{\underline{\omega}}$ which are homogeneous as in Equation \eqref{explicithomogeneitykappa} are identified with sections of $\underline{\omega}_\kappa$ as follows. The pullback of $\underline{\omega}$ to $T^{\times}_{\underline{\omega}}$ is canonically trivialized by $\alpha$, giving a trivialization $\alpha_\sigma$ of $\underline{\omega}_\sigma$ for each $\sigma$, and thus a trivialization $\alpha_\kappa$ of $\underline{\omega}_\kappa$ with $\alpha_\kappa(1) = \bigotimes_\sigma \alpha_\sigma(1)^{\otimes k_\sigma} \in \underline{\omega}_\kappa$. Acting on the trivialization by some $g \in \T(K)$ sends $\alpha_\sigma$ to $\sigma(g)\alpha_\sigma$, and $(g\alpha)_\kappa = \kappa(g)\alpha_\kappa$.

Take a section $v$ of $\underline{\omega}_\kappa$ defined on all of $X_{/K}$, and pull it back to a section over $T^\times_{\underline{\omega}}$. Its value $v(A,\alpha)$ at a point corresponding to a HBAV $A$ with a trivialization $\alpha$ of its cotangent bundle is a multiple $f(A,\alpha)$ of the canonical basis $\alpha_\kappa(1)$ for the line bundle $\underline{\omega}_\kappa$. Since $v$ was pulled back from a section defined over $X$, we have that $v(A,\alpha) = v(A,\alpha g)$ for any $g \in \T(K)$. But the canonical trivialization of the fiber of $\underline{\omega}_\kappa$ over $(A,\alpha g)$ is $\kappa(g)$ times that of the canonical trivialization over $(A,\alpha)$. Thus we have
\begin{equation}
	f(A,\alpha g) = \kappa(g\inv) f(A,\alpha).
\end{equation}
\end{rem}

\subsection{Comparison to Classical Theory} \label{holocomparison}
In the classical theory, Hilbert modular forms are presented as holomorphic functions satisfying some transformation conditions, which is part of why we call them holomorphic here. We compare our theory, which is based on geometry, to this one. It serves as a base for Section \ref{NEARLYComparison}, in which we give a similar comparison for nearly holomorphic Hilbert modular forms. For this section, we let $K = \R$.

Let $\h = \left\{ z \in \C \,\mid\, \operatorname{Im}(z) > 0 \right\}$ be the upper half-plane, and consider the product $\h_F = \prod_{\sigma\in I}\h = \{(z_\sigma)_\sigma \mid z_\sigma \in \h \text{ for all } \sigma \in I\}$. Then let $\Gamma$ be a congruence subgroup of $\SL_2(\mathcal{O}_F)$.\footnote{This is what we want for $G^*$. For $G$, elements of $\Gamma$ should have totally positive determinant, not just determinant $1$.} It acts on $\h_F$ by
\begin{equation}
	\gamma = \begin{pmatrix}
		a & b \\ c & d
	\end{pmatrix}, \qquad \gamma \cdot (z_\sigma)_\sigma = \left( \frac{\sigma(a)z_\sigma + \sigma(b)}{\sigma(c)z_\sigma + \sigma(d)} \right)_\sigma.
\end{equation}
A Hilbert modular form is a function on $\h_F$ which behaves well with respect to this action, as made precise in Definition \ref{DefnHoloHMF}.
\begin{defn} \label{DefnHoloHMF}
A Hilbert modular form of level $\Gamma$ and weight $(k_\sigma)_\sigma$ is a holomorphic function $f \colon \h_F \to \C$ such that for all $\gamma \in \Gamma$ and $\underline{z} = (z_\sigma)_\sigma \in \h_F$,
\begin{equation}
	f(\gamma \cdot \underline{z}) = \left(\prod_\sigma (\sigma(c)z_\sigma + \sigma(d))^{k_\sigma}\right) f(\underline{z}).
\end{equation}
\end{defn}

Let $\Gamma = \Gamma_1(N)$. We consider the non-compactified modular variety $Y$, the largest open subscheme of $X$ for which the pullback of the universal HBAV is an Abelian scheme. There is a morphism $\theta \colon \h_F \to T^\times_{\underline{\omega},/Y(\C)}$ sending a tuple $(z_\sigma)_\sigma$ to a principally polarized HBAV over $\C$ with a trivialization of its cotangent bundle. Note that $\mathcal{O}_F$ acts naturally on $\mathcal{O}_F \otimes_\Z \C \cong \prod_\sigma \C$, and let $L_{\underline{z}}$ be the lattice inside $\prod_\sigma \C$ generated as an $\mathcal{O}_F$-module by $(1)_\sigma$ and $(z_\sigma)_\sigma$. Then we write $\theta(\underline{z}) = (A_{\underline{z}},\iota_{\underline{z}},\psi_{\underline{z}},\lambda_{\underline{z}},\omega_{\underline{z}})$, where
\begin{enumerate}
	\item $A_{\underline{z}}$ is the complex Abelian variety whose $\C$-points are $(\mathcal{O}_F \otimes \C)/L_{\underline{z}}$,
	\item the action $\iota_{\underline{z}}$ of $\mathcal{O}_F$ on $A_{\underline{z}}$ is induced by its natural action on $\mathcal{O}_F \otimes_\Z \C$, which descends to the quotient because $\mathcal{O}_F L_{\underline{z}} \subset L_{\underline{z}}$,
	\item $\psi_{\underline{z}}$ includes $\mu_N \otimes \mathcal{D}_F\inv \cong \frac{1}{N}\mathcal{O}_F/\mathcal{O}_F$ into $(\mathcal{O}_F \otimes \C)/L_{\underline{z}}$ in the natural way,
	\item $\lambda_{\underline{z}}$ is induced by the pairing sending $(w_1,w_2) \in \mathcal{O}_F \otimes_\Z \C$ to $\operatorname{Trace}_{F/\Q}\operatorname{Im}(w_1\overline{w_2})$, and
	\item we trivialize $\Omega^1_{A_{\underline{z}}}$ by identifying the tangent space of $A_{\underline{z}}$ with $\mathcal{O}_F \otimes_\Z \C \cong \prod_\sigma \C$, and picking d$w \in \Omega^1_{A_{\underline{z}}}$ to be the $\C$-linear functional that sends $1$ in each component to $1$.
\end{enumerate}
Further, there is an action of $\Gamma = \Gamma_1(N)$ on $T^\times_{\underline{\omega},/Y(\C)}$ for which $\theta$ is equivariant. Thus we can pull back a complex Hilbert modular form, which is a function on $T_{\underline{\omega},/Y(\C)}^\times$, to a holomorphic function on $\h_F$. It is somewhat nontrivial to check that $\theta^*f$ satisfies the transformation property; we omit it here because we prove a more general fact in Theorem \ref{compareholothm}.

\subsection{Canonical Subgroups and Related Constructions}
Let $K$ be a $p$-adic field so we can base change to $S = \operatorname{Spec}\mathcal{O}_K$. We formally complete $X$ along $(p)$ to obtain a formal scheme $\mathfrak{X}$ with rigid analytic fiber $\mathcal{X}$. Inside of $\mathfrak{X}$, we have a formal open subscheme $\mathfrak{X}^{ord} \subset \mathfrak{X}$ which classifies \emph{ordinary} HBAVs. It is known as the ordinary locus, and its rigid analytic fiber is denoted $\mathcal{X}^{ord}$.

For each of the $g$ primes $\mathfrak{p}$ of $F$ lying over $p$, we have a Hodge height $\operatorname{Hdg}_\mathfrak{p}$ which measures how supersingular a HBAV is. It only depends on the substring $(A,\iota)$, with $\operatorname{Hdg}_\mathfrak{p}(A,\iota) \in \Q \cap [0,1]$. This Hodge height is the truncated valuation of the Hasse invariant, which is the determinant of verschiebung acting on $\underline{\omega}$. $A$ is ordinary at $p$ if and only if $\operatorname{Hdg}_\mathfrak{p}(A) = 0$ for all $\mathfrak{p}$ lying over $p$.

For any $\underline{v} = (v_1,\dots,v_g)$, we construct the $\underline{v}$-overconvergent locus $\mathfrak{X}(\underline{v})$, which is a normal formal scheme that classifies HBAVs $A$ with $\operatorname{Hdg}_{\mathfrak{p}_i}(A) < v_i$ for all $i$. This description gives a well-defined subset of the rigid analytic fiber $\mathcal{X}(\underline{v}) \subset \mathcal{X}$; we take a formal model $\mathfrak{X}^\prime(\underline{v})$ for it by taking admissible formal blowups as in \cite{AIP2} and normalize it in $\mathcal{X}(\underline{v})$ to obtain the formal overconvergent locus $\mathfrak{X}(\underline{v})$.

For each tuple $\underline{v}$ with $v_i < \frac{1}{p^n}$ for all $i$, we have a canonical subgroup $\mathfrak{C}_n \to \mathfrak{X}(\underline{v})$ of order $p^{nd}$. In fact, it is an $\mathcal{O}_F$-submodule of $\mathcal{A}[p^n]$. Its fiber over the ordinary locus is locally isomorphic to $\mu_{p^n} \otimes \mathcal{D}_F\inv$. Its fiber over the complement $\mathfrak{X}(\underline{v}) \setminus \mathfrak{X}^{ord}$ is isomorphic to a different formal group scheme, but both become isomorphic to each other and to the constant group scheme $\mathcal{O}_F/p^n\mathcal{O}_F$ at the rigid fiber.

We move to the rigid fiber $\mathcal{C}_n \to \mathcal{X}(\underline{v})$ in order to define $\mathcal{I}_n = \operatorname{Isom}_{\mathcal{X}(\underline{v})}(\mathcal{C}_n,\mu_{p^n} \otimes_\Z \mathcal{D}_F\inv)$. The map $h_n \colon \mathcal{I}_n \to \mathcal{X}(\underline{v})$ is a finite \'etale cover with Galois group $\T(\Z/p^n\Z)$. We then let $\mathfrak{I}_n$ be the formal model of $\mathcal{I}_n$ obtained by normalizing $\mathfrak{X}(\underline{v})$ inside $\mathcal{I}_n$, and we let $h_n$ also refer to the covering $\mathfrak{I}_n \to \mathfrak{X}(\underline{v})$. We call $\mathcal{I}_n$ and $\mathfrak{I}_n$ ``partial Igusa towers".

The canonical subgroups help us pick out an important subsheaf $\mathcal{F} \subset \underline{\omega}$, which is described in \cite[Proposition 3.4]{AIP2}.

\begin{prop}
There is a unique subsheaf $\mathcal{F}$ of $\underline{\omega}$ which is locally free of rank $1$ as a $\mathcal{O}_F \otimes_\Z \mathcal{O}_{\mathfrak{I}_n}$-module and contains $p^{\frac{\sup_i\{v_i\}}{p-1}}\underline{\omega}$. Moreover, for all $0 < w < n - \sup_i\{v_i\}\frac{p^n}{p-1}$, we have $\mathcal{O}_F$-linear maps
\begin{equation}
	HT_w \colon \mathfrak{C}_n^D \to h_n^*(\mathcal{F})/p^wh_n^*(\mathcal{F}).
\end{equation}
These induce isomorphisms
\begin{equation}
	HT_w \otimes 1 \colon \mathfrak{C}_n^D \otimes (\mathcal{O}_{\mathfrak{I}_n}/p^w\mathcal{O}_{\mathfrak{I}_n}) \to h_n^*(\mathcal{F})/p^wh_n^*(\mathcal{F}).
\end{equation}
\end{prop}

We refer to \cite[Proposition 3.4]{AIP2}, and \cite[Proposition 4.3.1]{AIP}, for a proof.

\begin{rem}
The fact that $\underline{\omega} \subset \mathcal{F} \subset p^{\frac{\sup_i\{v_i\}}{p-1}}\underline{\omega}$ shows that $\mathcal{F} \otimes_{\mathcal{O}_K} K = \underline{\omega} \otimes_{\mathcal{O}_K} K$ as families of modules. Their rigid fibers give different integral structures to the rigid analytification of $\underline{\omega}$, and only $\mathcal{F}$ gives a locally free sheaf of $\mathcal{O}_F$-modules. Specifically, the rigid fiber of $\underline{\omega}$ and the rigid fiber of $\mathcal{F}$ should be thought of as two versions of ``the ball of radius $1$" inside the analytification of $\underline{\omega}$ pulled back to $\mathcal{I}_n$.
\end{rem}

\subsection{Overconvergent Hilbert Modular Forms} \label{OverconvergentSection}
We fix a tuple $\underline{v}$ such that $v_i < \frac{1}{p^n}$ for all $i$. We have our canonical subgroup of level $p^n$, $\mathfrak{C}_n \to \mathfrak{X}(\underline{v})$, and a partial Igusa tower $\mathfrak{I}_n$.

We write $\operatorname{ev}_1^D \colon \mathcal{I}_n \to \mathcal{C}_n^D$ for the map sending $u$ to $u^D(1)$. This descends to a map of formal schemes $\operatorname{ev}_1^D \colon \mathfrak{I}_n \to \mathfrak{C}_n^D$ by the universal property of the relative normalization. Following \cite{AIP2}, we define the formal affine morphism $\gamma_w \colon \mathfrak{T}_{\mathcal{F},w}^\times \to \mathfrak{I}_n$, for any $n-1<w\leq n-\operatorname{sup}\{v_i\}\frac{p^n}{p-1}$, as follows. For every normal $p$-adically complete and separated, flat $\mathcal{O}_K$-algebra $R$ and for every morphism $\gamma \colon \operatorname{Spf}(R) \to \mathfrak{I}_n$, its $R$-valued points over $\gamma$ classify frames $\alpha \colon \gamma^*(\mathcal{F}) \to \mathcal{O}_F \otimes_\Z R$, such that $HT_w(\operatorname{ev}_1^D(u)) \equiv \alpha\inv(1) \pmod{p^w}$. The reference cited above writes this as needing to send $1$ to $1$ in the composite
\begin{equation}
	\mathcal{O}_F/p^n\mathcal{O}_F \xrightarrow{u^D} \mathfrak{C}_n^D(R) \xrightarrow{HT_w}\gamma^*\mathcal{F}/p^w\gamma^*(\mathcal{F}) \xrightarrow{\alpha} \mathcal{O}_F \otimes_\Z R/p^wR.
\end{equation}
The map $\gamma_w \colon \mathfrak{T}_{\mathcal{F},w}^\times \to \mathfrak{I}_n$ is a formal torsor for the group $\T_w^0$. Composing with the projection $\mathfrak{I}_n \to \mathfrak{X}(\underline{v})$ gives a formal torsor $\pr_w \colon \mathfrak{T}_{\mathcal{F},w}^\times \to \mathfrak{X}(\underline{v})$ for the group $\T_w$. The rigid fiber is denoted $\mathcal{T}_{\mathcal{F},w}^\times$.

\begin{rem}
In \cite{AIP2}, this is called $\mathfrak{IW}_w^+$. In other references, such as \cite{AIP,L1}, there are two related formal torsors: one is called $\mathfrak{IW}_w^+$, while the other is $\mathfrak{T}_{\mathcal{F},w}^\times$. In the case of Hilbert modular forms, the two coincide \--- we use the second name here so that our discussion in Section \ref{NOHMFs} better parallels that of \cite[Section 3.4]{L1}.
\end{rem}

For any representation $(\kappa,W_\kappa)$ of $\T_w$, we may construct a corresponding sheaf $\underline{\omega}_{\kappa,w}$ on $\mathcal{X}(\underline{v})$. Over an affinoid open $U = \operatorname{Sp}R \subset \mathcal{X}(\underline{v})$ when $\underline{\omega}|_{U}$ is trivial, sections of $\underline{\omega}_{\kappa,w}$ correspond to functions
\begin{equation}
	H^0(\mathcal{X}(\underline{v}),\underline{\omega}_\kappa) = \left\{ f \colon \mathcal{IW}_w^+(U) \to W_\kappa \otimes R \,\mid\, f(A,\alpha g) = g\inv\cdot f(A,\alpha) \right\}.
\end{equation}

\begin{defn}
Let $\kappa$ be a $w$-analytic weight. A $\underline{v}$-overconvergent Hilbert modular form of level $\Gamma_1(N)$ and weight $\kappa$ is a global section of $\underline{\omega}_{\kappa}$. The $K$-vector space of $\underline{v}$-overconvergent Hilbert modular forms of level $\Gamma_1(N)$ and weight $\kappa$ is thus $H^0(\mathcal{X}(\underline{v}),\underline{\omega}_{\kappa,w})$.

An overconvergent Hilbert modular form is a $\underline{v}$-overconvergent Hilbert modular form for some tuple $\underline{v}$ with each $v_i > 0$. The $K$-vector space of overconvergent Hilbert modular forms is the union $\bigcup_{\underline{v}}H^0(\mathcal{X}(\underline{v}),\underline{\omega}_{\kappa,w})$.
\end{defn}

\section{Nearly Hilbert Modular Forms}

\subsection{The de Rham Sheaf} Consider the relative de Rham cohomology $H^1_{dR}(\mathcal{A}/X) = R^1\pi_*\Omega^\bullet_{\mathcal{A}/X}$ of $\mathcal{A} \to X$. This sheaf has a natural subsheaf $\mathcal{H}$ which is locally free of rank $2$ as a sheaf of $\mathcal{O}_F \otimes_\Z \mathcal{O}_X$-modules;\footnote{In fact, $H^1_{dR}(\mathcal{A}/X)$ agrees with $\mathcal{H}$ over $Y$. This $\mathcal{H}$ is a canonical extension of $H^1_{dR}(\mathcal{A}/Y)$, a locally free sheaf of $\mathcal{O}_F \otimes_\Z \mathcal{O}_F$-modules of rank $2$, to such a sheaf defined over all of $X$.} over $K$, it fits into an exact sequence of sheaves of $\mathcal{O}_F \otimes_\Z \mathcal{O}_X$-modules known as the Hodge filtration:
\begin{equation}
	0 \to \underline{\omega} \to \mathcal{H} \to \underline{\omega}^\vee \to 0.
\end{equation}
Over $K$, $\underline{\omega}$ and $\underline{\omega}^\vee$ are locally free of rank $1$ as sheaves of $\mathcal{O}_F\otimes_\Z\mathcal{O}_{X}$-modules, so they are projective. Thus the sequence splits, though non-canonically. In addition, $\mathcal{H}$ admits a nondegenerate alternating pairing with respect to which $\underline{\omega}$ is a maximal totally isotropic subspace, which gives the explicit isomorphism between $\mathcal{H}/\underline{\omega}$ and $\underline{\omega}^\vee$.

Finally, we note that $\mathcal{H}$ splits as a direct sum $\mathcal{H} = \bigoplus_\sigma \mathcal{H}_\sigma$ as $\underline{\omega}$ did. This gives a similar exact sequence for each embedding $\sigma$:
\begin{equation}
	0 \to \underline{\omega}_\sigma \to \mathcal{H}_\sigma \to \underline{\omega}_\sigma^\vee \to 0.
\end{equation}

This sequence gives a filtration $\underline{\omega}_\sigma \subset \mathcal{H}_\sigma$ of $\mathcal{H}_\sigma$, which induces filtrations on $\operatorname{Sym}^k\mathcal{H}_\sigma$ for each $k$, with $\operatorname{Fil}^r\operatorname{Sym}^k\mathcal{H}_\sigma = \underline{\omega}_\sigma^{k-r} \otimes \operatorname{Sym}^r\mathcal{H}_\sigma$. We also get a filtration on $\operatorname{Sym}^{\kappa}\mathcal{H} \coloneqq \bigotimes_\sigma \operatorname{Sym}^{k_\sigma}\mathcal{H}_\sigma$ indexed by the partially ordered set $\Z_{\geq 0}[I]$, where a pure tensor $\bigotimes_\sigma s_\sigma$ is in $\operatorname{Fil}^\nu\operatorname{Sym}^{\kappa}\mathcal{H}$ if $s_\sigma \in \operatorname{Fil}^{\nu_\sigma}\operatorname{Sym}^{k_\sigma}\mathcal{H}_\sigma$ for all $\sigma$. In particular, $\operatorname{Fil}^0\operatorname{Sym}^{\kappa}\mathcal{H} = \underline{\omega}_\kappa$, where $\nu = 0$ means $\nu = (0)_\sigma$.

We define two $X$-schemes as follows. Consider the sheaf of modules $[\mathcal{O}_F \otimes_\Z \mathcal{O}_X]^{\oplus 2}$, with a filtration given by making the first copy of $\mathcal{O}_F \otimes \mathcal{O}_X$ serve as a submodule. Then let $T^{\times}_\mathcal{H} = \operatorname{Isom}_{X,F}([\mathcal{O}_F \otimes_\Z \mathcal{O}_X]^{\oplus 2},\mathcal{H})$ be the honest $\mathcal{O}_F$-frame bundle of $\mathcal{H}$, and $T^{\times,+}_\mathcal{H} = \operatorname{Isom}_{X,F}^+([\mathcal{O}_F \otimes_\Z \mathcal{O}_X]^{\oplus 2},\mathcal{H})$ be the filtered $\mathcal{O}_F$-frame bundle for $\mathcal{H}$, where the superscript $+$ indicates that the isomorphisms must respect the filtration. The structure map $\pi_\mathcal{H} \colon T^\times_\mathcal{H} \to X$ is a torsor for $G = \operatorname{Aut}([\mathcal{O}_F \otimes_\Z \mathcal{O}_X]^{\oplus 2}) = \operatorname{Res}_{\mathcal{O}_F/\Z}\GL_2$, while $\pi_\mathcal{H}^+ \colon T^{\times,+}_\mathcal{H} \to X$ is a torsor for the group $\operatorname{Aut}^+([\mathcal{O}_F \otimes_\Z \mathcal{O}_X]^{\oplus 2})$ of automorphisms that respect the filtration, which is the Borel subgroup $Q$ consisting of upper triangular matrices.

There are maps $T^{\times,+}_\mathcal{H} \to T^\times_\mathcal{H}$, as any frame of $\mathcal{H}$ that respects the filtration is a frame of $\mathcal{H}$; and $T^{\times,+}_\mathcal{H} \to T^\times_{\underline{\omega}}$, as the isomorphism restricted to the first step of the filtration is exactly a frame for $\underline{\omega}$.

We have a functor $\mathcal{E}$ from the category of algebraic representations of $Q$ to the category of coherent sheaves on $X$. Given a representation $V$, the sections of the associated sheaf $\mathcal{E}(V)$ are homogeneous functions to $V$. We can write this explicitly over an affine open $U = \operatorname{Spec}R \subset X$ when the restriction $\mathcal{H}|_U$ is trivial.
\begin{equation}
	H^0(U,\mathcal{E}(V)) = \left\{ f \colon \pi_{\mathcal{H}}^{+,-1}(U) \to V \otimes R \,\mid\, f(\alpha g) = g\inv \cdot f(\alpha) \right\}.
\end{equation}

\begin{rem}
The associated bundle is often written as a contracted product $\mathcal{E}(V) = T^{\times,+}_\mathcal{H} \times_Q V$. If $V$ is the restriction of a given representation of $G$, we may also form $T^\times_{\mathcal{H}} \times_G V$, which will agree with $\mathcal{E}(V)$. This is a general construction, which applies to a torsor for any group and any algebraic representation of that group.
\end{rem}

If $V$ is the standard representation of $\GL_2$, we can form the $G$-module $\operatorname{Sym}^\kappa V = \bigotimes_\sigma \operatorname{Sym}^{k_\sigma}V$ over $K$, where $g \in G(R)$ acts on a pure tensor $\bigotimes_\sigma v_\sigma$ by sending it to $\bigotimes_\sigma \sigma(g) \cdot v_\sigma$. Then $\mathcal{E}\left(\operatorname{Sym}^\kappa V\right) = \operatorname{Sym}^\kappa \mathcal{H}$. When the $G$-representation $\operatorname{Sym}^\kappa V$ is restricted to $Q$, is acquires a filtration for which $\operatorname{Fil}^\nu\operatorname{Sym}^\kappa V$ is stabilized by $Q$, and $\mathcal{E}\left(\operatorname{Fil}^\nu\operatorname{Sym}^\kappa V\right) = \operatorname{Fil}^\nu \operatorname{Sym}^\kappa \mathcal{H}$.

\begin{rem}\label{GversusBforWeights}
Since $\operatorname{Sym}^\kappa V$ is a $G$-module, we could have described $\operatorname{Sym}^\kappa \mathcal{H}$ as $T^\times_\mathcal{H} \times_G \operatorname{Sym}^\kappa V$ instead of using $T^{\times,+}_\mathcal{H}$. However, the same cannot be said for e.g. $\underline{\omega}_\kappa = \operatorname{Fil}^0\operatorname{Sym}^\kappa \mathcal{H}$. This is an important distinction. The $Q$-torsor $T^{\times,+}_\mathcal{H}$ is useful for describing weights \emph{and} types of nearly holomorphic Hilbert modular forms. However, the Gauss--Manin connection does {not} respect the filtration. While it does give a principal connection on $T^\times_\mathcal{H}$, and on any associated bundles thereof, it does not give one on $T^{\times,+}_\mathcal{H}$. This tension is why the theory of nearly holomorphic automorphic forms exists.
\end{rem}

\subsection{The Gauss--Manin Connection}
The relative de Rham cohomology admits a Gauss--Manin connection, which extends to a connection on $\mathcal{H}$ with logarithm poles on the boundary $C$. We denote it by $\nabla \colon \mathcal{H} \to \mathcal{H} \otimes \Omega^1_X(\log C)$. It induces a principal connection on $T^\times_\mathcal{H}$, and a connection on any associated bundle $T^\times_\mathcal{H} \times_G V$ for any representation $V$ of $G$.

We can describe it in terms of the action of the Lie algebra. Specifically, we let $U = \operatorname{Spec}R \subset X$ be an affine open for which $\mathcal{H}|_U$ is trivial, and we pick $D$ a derivation $R \to R$, viewing it as an element of $T_X(U)$.\footnote{The exterior derivative $d \colon R \to \Omega^1_U$ gives rise to a directional derivative $d(D) \colon R \to R$ for any $D \in T_X(U)$. The directional derivative in the direction of $D$ is a derivation on $R$, and every derivation arises this way.} Then the covariant derivative in the direction of $D$ is a linear map $\nabla(D) \colon \mathcal{H} \to \mathcal{H}$, which also commutes with the action of $\mathcal{O}_F$. Any frame $\alpha \in T^\times_\mathcal{H}$ gives a natural basis for the fibers, and $\nabla(D)$ is given in this basis by some matrix $X(D,\alpha) \in M_{2}(\mathcal{O}_F \otimes_\Z R)$ (if $U$ meets the boundary, it will actually be in $M_{2}(\mathcal{O}_F \otimes_\Z \operatorname{Frac}R)$ with logarithm poles over $C$). This is $\mathfrak{g}(R)$, where $\mathfrak{g}$ is the Lie algebra of $G$.

The map $\alpha \mapsto X(D,\alpha)$ is $G$-equivariant, in the sense that $X(D,\alpha g) = \operatorname{Ad}(g\inv)(X(D,\alpha))$ for all $g \in G$. This is the standard condition for $\nabla$ to induce a principal connection on $T^\times_\mathcal{H}$. One can see this formula at the level of linear algebra by using $g$ as a change of basis matrix, and noting that the $\operatorname{Ad}$ action of $G$ on $\mathfrak{g}$ is given by conjugation.

Let $V$ be a representation of $G$, so that it is also a representation of $\mathfrak{g}$. Then the associated bundle $T^\times_\mathcal{H} \times_G V$ acquires a connection whose covariant derivative $\nabla_V$ can be described as
\begin{equation}\label{CovariantDerivative}
	\nabla_V(f)(\alpha) = Df(\alpha) + X(D,\alpha) \cdot f(\alpha).
\end{equation}
The action on pure tensors $f \otimes r \in V \otimes R$ is $D(f \otimes r) = f \otimes D(r)$ and $X(D,\alpha) \cdot (f \otimes r) = (X(D,\alpha) \cdot f) \otimes r$.
\begin{rem}
One might think of functions to $V \otimes R$ as $R$-linear combinations of vectors in $V$. The $Df(\alpha)$ term in Equation \eqref{CovariantDerivative} differentiates the function $r \in R$ using the exterior derivative, while the $X(D,\alpha) \cdot f(\alpha)$ term differentiates the sections of $V$. This shows that it is a natural form for a connection that might be easier to describe in other ways.
\end{rem}

The Gauss--Manin connection is also used to define the Kodaira--Spencer map. Consider the following composition.
\begin{equation}\label{KS}
\underline{\omega} \otimes \underline{\omega} \xrightarrow{\Id \otimes \iota} \underline{\omega} \otimes \mathcal{H} \xrightarrow{\Id \otimes \nabla} \underline{\omega} \otimes \mathcal{H} \otimes \Omega^1_X(\log C) \xrightarrow{\Id \otimes \operatorname{pr} \otimes \Id} \underline{\omega} \otimes \underline{\omega}^\vee \otimes \Omega^1_X(\log C) \to \Omega^1_X(\log C).
\end{equation}
This morphism is surjective, and over $K$ it identifies $\Omega^1_X(\log C)$ with the summand $\bigoplus_\sigma \underline{\omega}_\sigma^{\otimes 2}$ of $\underline{\omega}^{\otimes 2}$. We write this isomorphism as $\operatorname{KS} \colon \Omega^1_X(\log C) \xrightarrow\sim \bigoplus_\sigma \underline{\omega}_\sigma^{\otimes 2}$. We will also use the projection $\operatorname{KS}_\sigma \colon \Omega^1_X(\log C) \to \underline{\omega}_\sigma^{\otimes 2}$.


\subsection{$(\mathfrak{g},Q)$-Modules}

As noted in Remark \ref{GversusBforWeights}, the Gauss--Manin connection does not respect the filtration of $\mathcal{H}$. This presents a problem, since the filtration is critical to describing the type, not just the weight, of a nearly holomorphic Hilbert modular form. Fortunately, there is a workaround, which we take from \cite[Section 2.2]{L1}.\footnote{The ideas seem to have existed before Liu's paper. Liu does not claim to be the original source, pointing instead to \cite[Section VI.4]{FaltingsChai}. However, Liu's paper is the only source the author knows of which defines $(\mathfrak{g},Q)$-modules explicitly for use in the context of automorphic forms, so we cite it here.}

Let $\mathfrak{g} = \operatorname{Lie}G$ and $\mathfrak{q} = \operatorname{Lie}Q$ be the Lie algebras of $G$ and of its Borel subgroup $Q$, respectively.
\begin{defn}
A $(\mathfrak{g},Q)$-module defined over an algebra $E$ is an algebraic representation $V$ of both $\mathfrak{g}$ and $Q$ on locally free $E$-modules, satisfying the following compatibility conditions.
\begin{enumerate}
	\item The action of $\mathfrak{q} \subset \mathfrak{g}$ is the same as that induced by $Q$.
	\item For any $g \in Q$, $X \in \mathfrak{g}$, and $v \in V$,
	\begin{equation}
		\operatorname{Ad}(g)(X) \cdot v = g \cdot X \cdot g\inv \cdot v.
	\end{equation}
\end{enumerate}
\end{defn}

Any representation $V$ of $G$ has a natural action of $\mathfrak{g}$ and of $Q$. These actions will be compatible in both senses above. In fact, any finite rank $(\mathfrak{g},Q)$-module arises this way, c.f. \cite[Remark 2.7]{L1}.

In order to produce a $(\mathfrak{g},Q)$-module, one can start with a representation $V$ of $Q$ and form a basis of $\mathfrak{g}$ whose first elements are a basis of $\mathfrak{q} \subset \mathfrak{g}$. The elements of $\mathfrak{q}$ act in the way induced by $Q$, so one only needs to give formulas for the action of the rest in a compatible way.

\begin{rem} \label{gQmodulesConceptual}
Another way to describe a $(\mathfrak{g},Q)$-module $V$ is to pick some ring $R$ and extend the $Q(R)$-module $V(R)$ to a representation $\widehat{V}(R)$ of an open subgroup $\widehat{Q}(R) \subset G(R)$ containing $Q(R)$. Since $\widehat{Q}(R)$ is open in $G(R)$, its Lie algebra is $\mathfrak{g}$. One then checks that the action of $\mathfrak{g}$ preserves $V(R)$ as a subspace of $\widehat{V}(R)$, and then calculates the formulas for the action, which will automatically be compatible. These formulas will only involve scalars from $R$, so they will define a $(\mathfrak{g},Q)$-module over $R$. If the scalars all come from a subring of $R$ and the basis is defined over that subring, then the formulas will define a $(\mathfrak{g},Q)$-module over that subring. This is done e.g. in \cite[Section 2.3]{L1} (``a more conceptual proof" after Equation (2.5)) for $R = \Z_p$; we will use this method in Section \ref{NEARLY} as well.
\end{rem}

Let $V$ be a $(\mathfrak{g},Q)$-module. The formula in Equation \eqref{CovariantDerivative} shows that the Gauss--Manin connection on the bundle associated to a $G$-module is given in terms of the exterior derivative and the action of $\mathfrak{g}$. Thus one might hope that we can define a connection on $\mathcal{V} \coloneqq T^{\times,+}_\mathcal{H} \times_Q V$. This is most of the content of the following proposition.

\begin{prop}\label{ConnectionOnAssociatedBundleWithgAction}
Let $V$ be a $(\mathfrak{g},Q)$-module. Then there is an integrable connection $\nabla_\mathcal{V}$ on the associated bundle $\mathcal{V} \coloneqq T^{\times,+}_\mathcal{H} \times_Q V$ given by the formula
\begin{equation}
	\nabla_\mathcal{V}(D)(v)(\alpha) = Dv(\alpha) + X(D,\alpha) \cdot v(\alpha).
\end{equation}
\end{prop}
\begin{proof}
The formula certainly gives a map $\mathcal{O}_{T^{\times,+}_\mathcal{H}} \to \mathcal{O}_{T^{\times,+}_\mathcal{H}} \otimes \Omega^1_X(\log C)$. In order to show that it gives a connection on $\mathcal{V}$, we have to show that the homogeneity property holds, $\nabla_\mathcal{V}(v)(\alpha g) = g\inv \cdot \nabla_\mathcal{V}(v)(\alpha)$. In fact, it holds for each term separately, and we treat them separately. First,
\begin{equation}
	Dv(\alpha g) = D(g\inv \cdot v(\alpha)) = g \inv \cdot Dv(\alpha).
\end{equation}
The first equality is by the homogeneity of $v$, and the second by the chain rule for the exterior derivative $d$. Then,
\begin{equation}
	X(D,\alpha g) \cdot v(\alpha g) = \operatorname{Ad}(g\inv)(X(D,\alpha)) \cdot g\inv \cdot v(\alpha) = g\inv \cdot X(D,\alpha) \cdot g \cdot g\inv \cdot v(\alpha).
\end{equation}
This reduces to $g\inv \cdot X(D,\alpha) \cdot v(\alpha)$. Here we have used the second compatibility condition for $(\mathfrak{g},Q)$-modules in the second equality. Thus we have
\begin{equation}
\begin{array}{c c l}
	\nabla_\mathcal{V}(D)(v)(\alpha g) &=& Dv(\alpha g) + X(D,\alpha g) \cdot v(\alpha g) \\ &=& g\inv \cdot Dv(\alpha) + g\inv \cdot X(D,\alpha) \cdot v(\alpha) \\ &=& g\inv \cdot \nabla_\mathcal{V}(v)(\alpha).
\end{array}
\end{equation}
This is what we wanted to show. The connection is integrable because $d$ and $\nabla$ both are.
\end{proof}

\subsection{Nearly Holomorphic Hilbert Modular Forms} \label{NEARLY}
Nearly holomorphic Hilbert modular forms are sections of certain vector bundles. In particular, there should be finite rank representations $V_\kappa^\nu$ of $Q$ for all $\kappa \in \Z[I]$ and ${\nu} = (r_\sigma)_\sigma \in \Z_{\geq 0}[I]$ so that a nearly holomorphic Hilbert modular form of weight $\kappa$ and type ${\nu}$ is a section of $\mathcal{V}_\kappa^{\nu} = T^{\times,+}_\mathcal{H} \times_Q V_\kappa^{{\nu}}$. Each individual $V_\kappa^{{\nu}}$ is not a $(\mathfrak{g},Q)$-module, which complicates the definition of the differential operators. However, we will have inclusion $V_\kappa^{{\nu}_{1}} \subset V_\kappa^{{\nu}_{2}}$ whenever ${\nu_{i}} = ({r}_\sigma^{(i)})_\sigma$ and $r_\sigma^{(2)} \geq r_{\sigma}^{(1)}$ for all $\sigma$, and the union $V_\kappa = \bigcup_{\nu} V_\kappa^{\nu}$ is a $(\mathfrak{g},Q)$-module. It will satisfy $\mathfrak{g}_\sigma V_\kappa^{\nu} \subset V_\kappa^{{\nu} + \sigma}$. Thus, by Proposition \ref{ConnectionOnAssociatedBundleWithgAction} and the forthcoming Remark \ref{RaiseTypeBySigma}, we will obtain a connection $\nabla_{\mathcal{V}_\kappa}$ on $\mathcal{V}_\kappa = \bigcup_{\nu} \mathcal{V}_\kappa^{\nu}$ satisfying
\begin{equation}
\nabla_{\mathcal{V}_\kappa}(\mathcal{V}_\kappa^{\nu}) \subset \bigoplus_\sigma\mathcal{V}_\kappa^{{\nu}+\sigma} \otimes \underline{\omega}_\sigma^{\otimes 2} \subset \mathcal{V}_\kappa \otimes \Omega^1_X(\log C)
\end{equation}
The specifics are laid out in the rest of the section.


For now, we fix an auxilliary prime $\ell$ which is unramified in $F/\Q$, $L$ an $\ell$-adic field which splits $F$, and $\mathcal{O}_L$ its ring of integers. We let $I$ denote the set of embeddings of $F$ into $L$.

We fix a weight $\kappa = \sum_\sigma k_\sigma\sigma \in \Z[I]$, with $W_\kappa$ the corresponding $1$-dimensional representation of $\T$ defined over $\mathcal{O}_L$. We build the $(\mathfrak{g},Q)$-module $V_\kappa$ following the course laid out in Remark \ref{gQmodulesConceptual}.

We begin by giving names to important, non-algebraic subgroups of $G(\mathcal{O}_L)$. We will use the neighborhood $I_G(\mathcal{O}_L) \supset Q(\mathcal{O}_L)$, which contains the lower parabolic subgroup $Q_{I_G}^-(\mathcal{O}_L)$ and the Levi subgroup $H(\mathcal{O}_L)$,
\begin{equation}
\begin{array}{c}
	I_G(\mathcal{O}_L) = \left\{ \begin{pmatrix} a & b \\ c & d \end{pmatrix} \in G(\mathcal{O}_L) \mid c \in p\mathcal{O}_F \otimes \mathcal{O}_L \right\}, \\\\
	Q_{I_G}^-(\mathcal{O}_L) = \left\{ \begin{pmatrix} a & 0 \\ c & d \end{pmatrix} \in G(\mathcal{O}_L) \mid c \in p\mathcal{O}_F \otimes \mathcal{O}_L \right\}, \\\\
	H(\mathcal{O}_L) = \left\{ \begin{pmatrix} a & 0 \\ 0 & d \end{pmatrix} \in G(\mathcal{O}_L) \right\}.
\end{array}
\end{equation}
Note that $Q_{I_G}^-(\mathcal{O}_L)$ projects onto $\T(\mathcal{O}_L)$ by picking out the top left entry $a$. We inflate the representation $W_\kappa(L)$ from $\T(\mathcal{O}_L)$ to $Q_{I_G}^-(\mathcal{O}_L)$. We have a unique choice of $\ell$-adic topology on the finite dimensional Banach space $W_\kappa(L)$; consider the $\ell$-adic analytic induction $\widehat{V}_\kappa(\mathcal{O}_L) = \operatorname{Ind}_{Q_{I_G}^-(\mathcal{O}_L)}^{I_G(\mathcal{O}_L)}W_\kappa(L)$, which is the set of $\ell$-adic analytic functions $\phi \colon I_G(\mathcal{O}_L) \to W_\kappa(L)$ which are homogeneous with respect to the action of $Q_{I_G}^-(\mathcal{O}_L)$ in the sense that $\phi(hx) = h \cdot \phi(x)$ for all $x \in I_G(\mathcal{O}_L)$, $h \in Q_{I_G}^-(\mathcal{O}_L)$. It is a representation of $I_G(\mathcal{O}_L)$ by right translation: $(g \cdot \phi)(x) = \phi(xg)$ for all $g,x \in I_G(\mathcal{O}_L)$.

By the Iwahori decomposition, every coset $Q_{I_G}^-(\mathcal{O}_L)x \in Q_{I_G}^-(\mathcal{O}_L)\backslash I_G(\mathcal{O}_L)$ can be written as
\begin{equation}
	Q_{I_G}^-(\mathcal{O}_L)x = Q_{I_G}^-(\mathcal{O}_L)\begin{pmatrix}
		1 & \underline{Y} \\ 0 & 1
	\end{pmatrix}.
\end{equation}
This is for some unique choice of $\underline{Y} \in \mathcal{O}_F \otimes \mathcal{O}_L$, so that each $\phi \in \widehat{V}_\kappa$ is determined by its values on these matrices. We note that the coset spaces $Q_{I_G}^-(\mathcal{O}_L)\backslash I_G(\mathcal{O}_L)$ and $H(\mathcal{O}_L)\backslash Q(\mathcal{O}_L)$ are the same, down to the choice of representatives, and view $\widehat{V}_\kappa$ as an algebraic representation defined over $\mathcal{O}_L$ which consists of $\ell$-adic analytic functions from the coset space $H\backslash Q$ to $W_\kappa$. This allows us to write $\widehat{V}_\kappa(R) = W_\kappa \otimes \mathcal{O}_{H\backslash Q}^{rig}$, where $\mathcal{O}_{H \backslash Q}^{rig}$ is the space of $\ell$-adic analytic functions on the rigid space $H \backslash Q$.

For any $\mathcal{O}_L$-algebra $R$, we have natural coordinates on $(H\backslash Q)(R) \cong \mathcal{O}_F \otimes R \cong \prod_\sigma R$ called $Y_\sigma$ for each embedding $\sigma \colon F \to L$.\footnote{This is where we use the fact that our auxilliary prime $\ell$ is unramified.} By standard theory of rigid spaces, the ring of $\ell$-adic analytic functions on $H\backslash Q$ is $\mathcal{O}_{H \backslash Q}^{rig}(\operatorname{Sp}R) = R\langle \underline{Y} \rangle$, the space of power series over $R$ in the variables $Y_\sigma$ whose coefficients go to $0$ $\ell$-adically.

One may explicitly compute the action of $I_G(\mathcal{O}_L)$ on $\widehat{V}_\kappa(\mathcal{O}_L)$, obtaining the formulas
\begin{equation}
	(g \cdot P)(\underline{Y}) = (a+\underline{Y}c) \cdot P((a+\underline{Y}c)\inv(b+\underline{Y}d)), \qquad g = \begin{pmatrix}
		a & b \\ c & d \end{pmatrix}
\end{equation}
Here $\underline{Y}$ and the entries of $g$ are viewed as elements of the ring $\mathcal{O}_F \otimes \mathcal{O}_L$. The induced action of $Q(\mathcal{O}_L)$ and $\mathfrak{g}(\mathcal{O}_L)$ may be computed from this formula. For $Q(\mathcal{O}_L)$, this is simple: substitute $0$ for $c$ in the formula above. This gives the induced action of $\mathfrak{q}(\mathcal{O}_L)$ on $\widehat{V}_\kappa(\mathcal{O}_L)$. To describe the action of the rest of $\mathfrak{g}(\mathcal{O}_L)$, we need a basis. Since $\mathfrak{g}(\mathcal{O}_L) = \prod_\sigma \mathfrak{gl}_2(\mathcal{O}_L)$, we may specify the actions of $\{\mu_\sigma^- \mid \sigma \colon F \to L\}$, where $\mu_\sigma^-$ is the element of $\prod \mathfrak{gl}_2(\mathcal{O}_L)$ which is $\begin{pmatrix} 0 & 0 \\ 1 & 0 \end{pmatrix}$ in the entry corresponding to $\sigma$, and the zero matrix in all other entries. With our fixed choice of coordinates $Y_\sigma$, we find that
\begin{equation}
	\mu_\sigma^- \cdot P(\underline{Y}) = Y_\sigma \varepsilon_\sigma \cdot P(\underline{Y}) - Y_\sigma^2 \frac{\partial}{\partial Y_\sigma} P(\underline{Y}).
\end{equation}
Here $\varepsilon_\sigma \in \operatorname{Lie}(\T) = \prod \mathfrak{gl}_1(\mathcal{O}_L)$ is the tuple with a $1$ in the entry corresponding to $\sigma$ and $0$ in all other entries, acting naturally on $W_\kappa$. Specifically, $\varepsilon_\sigma \cdot w = k_\sigma w$ for any $w \in W_\kappa$.

Notice that the space of polynomials $W_\kappa[\underline{Y}]$ is preserved by the actions of both $Q(\mathcal{O}_L)$ and $\mathfrak{g}(\mathcal{O}_L)$, though it is not preserved by $I_G(\mathcal{O}_L)$. These are the algebraic functions on the scheme $H \backslash Q$. We have thus defined a $(\mathfrak{g},Q)$-module structure on $V_\kappa = W_\kappa \otimes \mathcal{O}_{H \backslash Q}$, where now $\mathcal{O}_{H \backslash Q}$ is the space of algebraic functions on the coset space $H \backslash Q$. Over $\mathcal{O}_L$, this is the polynomial ring $\mathcal{O}_{H \backslash Q}(R) = \mathcal{O}_L[\underline{Y}] = \mathcal{O}_L[Y_\sigma]_\sigma$.

Now we fix an algebra $R$ over the ring of integers in the Galois closure of $F$ such that the discriminant of $F$ is invertible in $R$. Over $R$, we may write $V_\kappa = W_\kappa \otimes \mathcal{O}_{H \backslash Q} = W_\kappa[\underline{Y}]$ where $\underline{Y} = (Y_\sigma)_\sigma$ is our natural set of coordinates, and we can extend any basis for $\mathfrak{q}$ to a basis for $\mathfrak{g}$ by using $\{\mu_\sigma^-\}$. We have actions given by formulas. For $Q$,
\begin{equation}
	(g \cdot P)(\underline{Y}) = (a \cdot P(a\inv(b+\underline{Y}d)), \qquad g = \begin{pmatrix}
		a & b \\ 0 & d \end{pmatrix}.
\end{equation}
For $\mathfrak{g}$, we specify the action of the $\mu_\sigma^-$.
\begin{equation}\label{actionofmusigmaminus}
	\mu_\sigma^- \cdot P(\underline{Y}) = Y_\sigma \varepsilon_\sigma \cdot P(\underline{Y}) - Y_\sigma^2 \frac{\partial}{\partial Y_\sigma} P(\underline{Y}).
\end{equation}
These formulas give compatible actions since they were compatible on $\widehat{V}_\kappa$.

\begin{rem}
If we replace $R$ by a subring $\mathcal{O} \subset R$, then $\mathfrak{g}(\mathcal{O})$ is a subalgebra of $\mathfrak{g}(R) = \prod_\sigma \mathfrak{gl}_2(R)$. Thus these formulas can still be used to describe the action, but we have to be careful about what elements of the direct product we consider. This is particularly important if $p$ ramifies in $F$ when $R = K$ is our chosen $p$-adic field and $\mathcal{O} = \mathcal{O}_K$ is its ring of integers.
\end{rem}

Not only does the action of $Q$ preserve the degree of the polynomial, it preserves the degree in each variable $Y^\sigma$ separately. Thus $V_\kappa$ has a filtration indexed by the partially ordered set $\Z_{\geq 0}[I]$, such that $V_\kappa^{\nu}$ is the space of polynomials of degree at most $r_\sigma$ in the variable $Y_\sigma$. The Lie algebra $\mathfrak{q}$ also preserves the filtration. We write $\mathfrak{g} = \bigoplus_\sigma \mathfrak{g}_\sigma$ over $K$. Since $\mathfrak{g}_\sigma$ is generated by $\mathfrak{q}_\sigma$ and $\mu_\sigma^-$, it preserves the degree as a function of $Y_\tau$ for each $\tau \neq \sigma$, and raises the degree as a function of $Y_\sigma$ by at most $1$ by the formula in Equation \eqref{actionofmusigmaminus}. We have $\mathfrak{g}_\sigma V_\kappa^{\nu} \subset V_\kappa^{\nu+\sigma}$.

\begin{defn}
The sheaf of nearly holomorphic Hilbert modular forms of weight $\kappa$ is $\mathcal{V}_\kappa = T^{\times,+}_\mathcal{H} \times_Q V_\kappa$. The sheaf of nearly holomorphic Hilbert modular forms of weight $\kappa$ and type ${\nu}$ is $\mathcal{V}_\kappa^{\nu} = T^{\times,+}_\mathcal{H} \times_Q V_\kappa^{\nu}$. The $K$-vector space of nearly holomorphic Hilbert modular forms of weight $\kappa$ and type $\nu$ is thus $H^0(X,\mathcal{V}_\kappa^\nu)$.
\end{defn}

\begin{rem} \label{RaiseTypeBySigma}
Using the filtration, we can refine the statement of Proposition \ref{ConnectionOnAssociatedBundleWithgAction}. Define the differential operator $\nabla_\sigma$ to be the composition
\begin{equation}
	\mathcal{V}_\kappa \xrightarrow{\nabla_{\mathcal{V}_\kappa}} \mathcal{V}_\kappa \otimes \Omega^1_X(\log C) \xrightarrow{1 \otimes KS_\sigma} \mathcal{V}_\kappa \otimes \underline{\omega}_\sigma^{\otimes 2} \cong \mathcal{V}_{\kappa + 2\sigma}.
\end{equation}
Then $\nabla_\sigma(D)$ is given by the action of $X(D_{\sigma},\alpha) \in \mathfrak{g}_\sigma$, where $D_{\sigma}$ is the projection of $D\in \Omega^1_X(\log C)$ onto the summand $\underline{\omega}_\sigma^{\otimes 2}$. Thus it sends $\mathcal{V}_\kappa^{\nu}$ into $\mathcal{V}_\kappa^{{\nu} + \sigma} \otimes \underline{\omega}_\sigma^{\otimes 2} \cong \mathcal{V}_{\kappa + 2\sigma}^{{\nu} + \sigma}$. Specifically, $\nabla_\sigma$ raises the weight of a nearly holomorphic Hilbert modular form by $2\sigma$ and its type by $\sigma$.

Note that we have written $\nabla_\sigma$ with no reference to $\kappa$ or ${\nu}$, since they are clear from context while $\sigma$ is not. In addition, we write $\kappa = \sum_\sigma k_\sigma\sigma$ in order to make better sense of what character $\kappa + 2\sigma$ corresponds to.
\end{rem}

It is known (c.f. \cite[Lemma 2.1.14]{Katz}) that $\nabla_\sigma$ and $\nabla_\tau$ commute for any pair of embeddings $\sigma$ and $\tau$. Thus we may unambiguously write $\nabla_{\kappa^\prime}$ for the differential operator that raises weights by $\kappa^\prime = \sum_\sigma 2k_\sigma^\prime\sigma$ and types $(k_\sigma)_\sigma$.

\subsection{Comparison to Classical Theory}\label{NEARLYComparison}
Let $\h_F$ be as in Section \ref{holocomparison}, and let $s_\sigma \colon \h_F \to \C$ be the smooth function $s_\sigma = s_\sigma(\underline{z}) = \frac{1}{z_\sigma - \overline{z}_\sigma}$. Following \cite[Section 13.11]{Shi00} we define
\begin{defn}\label{ShiNHHMF}
A nearly holomorphic Hilbert modular form of level $\Gamma$, weight $\kappa = (k_\sigma)_\sigma$, and type ${\nu} = (r_\sigma)_\sigma$ is a $C^\infty$ function $f \colon \h_F \to \C$ such that
\begin{enumerate}
	\item $f(\underline{z})$ is \emph{nearly holomorphic}, i.e., it can be written as a polynomial in the variables $s_\sigma$ whose coefficients are holomorphic functions $\h_F \to \C$ and whose degree as a function of only $s_\sigma$ is at most $r_\sigma$, and
	\item for all $\gamma \in \Gamma$ and $\underline{z} = (z_\sigma)_\sigma \in \h_F$, we let $\mu(\gamma,\underline{z}) = (\sigma(c)z_\sigma + \sigma(d))_\sigma \in \prod_\sigma \G_m(\C) \cong \T(\C)$, and $f$ satisfies the transformation property
	\begin{equation}
		f(\gamma \cdot \underline{z}) = \kappa(\mu(\gamma,\underline{z})) f(\underline{z}).
	\end{equation}
\end{enumerate}
\end{defn}
Let $s = s(\underline{z}) = \left(s_\sigma(\underline{z})\right)_\sigma$. For $\kappa = (k_\sigma)_\sigma$, we may write $\kappa(s) = \prod_\sigma s_\sigma^{k_\sigma}$ by viewing $s$ as an element of $\prod_\sigma \G_m(\C) \cong \T(\C)$. The Maass--Shimura differential operators are defined to be the composition
\begin{equation}\label{definemaassshimura}
	D_\kappa^\sigma(f) = \kappa(s) \frac{\partial}{\partial z_\sigma}\left[\kappa(s\inv)f \right].
\end{equation}
It is classical that $D_\kappa^\sigma$ raises the weight by $2\sigma$ and the type by $\sigma$.

We have defined nearly holomorphic Hilbert modular forms in terms of functions on $T^{\times,+}_{\mathcal{H}}$. In order to pull these back to nearly holomorphic functions on $\h_F$, we need to extend our map $\h_F \to T^\times_{\underline{\omega},/Y(\C)}$ from Section \ref{holocomparison} to a map $\h_F \to T^{\times,+}_{\mathcal{H},/Y(\C)}$.

For $\underline{z} = (z_\sigma)_\sigma \in \h_F$, we get the Abelian variety $A_{\underline{z}} = (\mathcal{O}_F \otimes_\Z \C)/L_{\underline{z}}$. Its tangent space is $\mathcal{O}_F \otimes_\Z \C \cong \prod_\sigma \C$, which has a natural $\R$-basis: the copy of $\C$ corresponding to $\sigma$ is spanned by $1$ and $z_\sigma$. The de Rham cohomology $H^1_{dR}(A_{\underline{z}})$ is the complexification of the cotangent bundle, so this basis of the tangent space gives a dual basis $\{\alpha_\sigma,\beta_\sigma \mid \sigma \in I\}$ for the de Rham cohomology. For any $(a_\sigma + b_\sigma z_\sigma) \in \prod_\sigma \C$,
\begin{equation}
	\alpha_\sigma\left(  (a_{\sigma} + b_\sigma z_\sigma)_\sigma \right) = a_\sigma \qquad \text{and} \qquad \beta_\sigma\left(  (a_{\sigma} + b_\sigma z_\sigma)_\sigma \right) = b_\sigma.
\end{equation}
This $\C$-basis for $H^1_{dR}(A_{\underline{z}})$ gives rise to a $\mathcal{O}_F \otimes_\Z \C$-basis $(\alpha,\beta)$ for the de Rham cohomoogy. It is horizontal for the Gauss--Manin connection, $\nabla(\alpha_\sigma) = \nabla(\beta_\sigma) = 0$. However, it does not respect the filtration, so it does not determine an element of $T^{\times,+}_\mathcal{H}$. Instead we take d$w$ from Section \ref{holocomparison}, which trivializes $\underline{\omega}$, and $\beta$ above. This $(\text{d}w,\beta)$ is a holomorphic frame for $H^1_{dR}(A_{\underline{z}})$ as a $\mathcal{O}_F \otimes \C$-module.

As sheaves of $C^\infty(\h_F,\C)$-modules, there is another basis $(\text{d}w,\text{d}\overline{w})$ for $H^1_{dR}(A_{\underline{z}})$. This basis does not give a suitable element of $T^{\times,+}_{\mathcal{H}}$, but it does if we replace it by $(\text{d}w,-\text{d}\overline{w}\,s)$. Our two bases are related by
\begin{equation} \label{relatebasesofderham}
	(\text{d}w,-\text{d}\overline{w}\,s) = (\text{d}w,\beta) \begin{pmatrix}
		1 & -s \\ 0 & 1
	\end{pmatrix}.
\end{equation}

We define a map $\phi$ from global sections of $\mathcal{V}_{\kappa}^{\nu}$ to nearly holomorphic Hilbert modular forms in the sense of Definition \ref{ShiNHHMF},
\begin{equation}
	\phi(f)(\underline{z}) = f(A_{\underline{z}},\iota_{\underline{z}},\psi_{\underline{z}},\lambda_{\underline{z}},\text{d}w, -\text{d}\overline{w}\,s)|_{\underline{Y} = 0}.
\end{equation}

We may use Equation \eqref{relatebasesofderham} to get another description. Since $(\text{d}w,\beta)$ is a holomorphic frame, we may write $f(A_{\underline{z}},\iota_{\underline{z}},\psi_{\underline{z}},\lambda_{\underline{z}},\text{d}w,\beta) = P_f(\underline{Y})$ for some polynomial $P_f$ of degree at most $r_\sigma$ in each $Y_\sigma$ whose coefficients vary holomorphically with $\underline{z}$. Then,
\begin{equation}
\begin{array}{r c l}
	\phi(f)(\underline{z}) &=& f(A_{\underline{z}},\iota_{\underline{z}},\psi_{\underline{z}},\lambda_{\underline{z}},\text{d}w,-\text{d}\overline{w}\,s)|_{\underline{Y} = 0} \\\\
	&=& f\left(A_{\underline{z}},\iota_{\underline{z}},\psi_{\underline{z}},\lambda_{\underline{z}},(\text{d}w,\beta)\begin{pmatrix}
		1 & -s \\ 0 & 1
	\end{pmatrix} \right)|_{\underline{Y} = 0} \\\\
	&=& \begin{pmatrix}
		1 & s \\ 0 & 1
	\end{pmatrix} \cdot f(A_{\underline{z}},\iota_{\underline{z}},\psi_{\underline{z}},\lambda_{\underline{z}},\text{d}w,\beta)|_{\underline{Y} = 0} \\\\
	&=& \begin{pmatrix}
		1 & s \\ 0 & 1
	\end{pmatrix} \cdot P_f(\underline{Y})|_{\underline{Y} = 0} \\\\
	&=& P_f(\underline{Y} + s)|_{\underline{Y} = 0} = P_f(s(\underline{z})).
\end{array}
\end{equation}
So $\phi(f)(\underline{z}) = f(A_{\underline{z}},\iota_{\underline{z}},\psi_{\underline{z}},\lambda_{\underline{z}},\text{d}w,-\text{d}\overline{w}\,s)|_{\underline{Y} = 0} = f(A_{\underline{z}},\iota_{\underline{z}},\psi_{\underline{z}},\lambda_{\underline{z}},\text{d}w,\beta)|_{\underline{Y} = s(\underline{z})} = P_f(s(\underline{z}))$. Since $P_f$ is a polynomial of degree at most $r_\sigma$ in the $Y_\sigma$'s whose coefficients are holomorphic functions on $\h_F$, $\phi(f) = P_f(s(\underline{z}))$ is such a polynomial in the $s_\sigma$'s whose coefficients are such holomorphic functions $\h_F \to \C$. This map leads to two comparison theorems.
\begin{thm}\label{compareholothm}
Let $N^{\nu}_\kappa$ be the space of nearly holomorphic Hilbert modular forms of level $\Gamma_1(N)$, weight $\kappa$, and type $\nu$ in the sense of Definition \ref{ShiNHHMF}. Then $\phi \colon H^0(X,\mathcal{V}_\kappa^{\nu}) \to N_\kappa^{\nu}$ is an isomorphism.
\end{thm}
\begin{proof}
Let $f$ be a global section of $\mathcal{V}_\kappa^{\nu}$. By the discussion immediately preceding this theorem, $\phi(f)$ is written as a polynomial of the correct degree in the $s_\sigma$'s whose coefficients are holomorphic maps $\h_F \to \C$. The correspondence $P_f(\underline{Y}) \mapsto P_f(s)$ is easily reversed, so we just need to show that the transformation property is preserved. For $\underline{z} = (z_\sigma)_\sigma \in \h_F$ and $\gamma \in \Gamma$, we let $\mu(\gamma,\underline{z}) \in \T(\C)$ be
\begin{equation}
	\mu(\gamma,\underline{z}) = (\sigma(c)z_\sigma + \sigma(d))_\sigma \in \prod_\sigma \G_m(\C) \cong \T(\C), \qquad \gamma = \begin{pmatrix}
		a & b \\ c & d
	\end{pmatrix}.
\end{equation}
Then we want to show that $\phi(f)(\gamma \cdot \underline{z}) = \kappa(\mu(\gamma,\underline{z}))\phi(f)(\underline{z}) = \mu(\gamma,\underline{z}) \cdot \phi(f)(\underline{z})$. The map $\gamma$ gives an isomorphism $A_{\underline{z}} \xrightarrow{\sim} A_{\gamma \cdot \underline{z}}$, $w \mapsto \mu(\gamma,\underline{z})\inv w$. This modifies the frames for the de Rham cohomology
\begin{equation}
	\gamma^*(\text{d}w,-\text{d}\overline{w}\,s) = (\text{d}w,-\text{d}\overline{w}\,s)\begin{pmatrix}
		\mu(\gamma,\underline{z})\inv & 0 \\ 0 & \mu(\gamma,\underline{z})
	\end{pmatrix}.
\end{equation}
From the action of $Q$ on $V_\kappa^{\nu}$, we get the desired transformation property 
\begin{equation}
\begin{array}{r c l}
	\phi(f)(\gamma \cdot \underline{z}) &=& f(A_{\gamma \cdot \underline{z}},\iota_{\gamma \cdot \underline{z}},\psi_{\gamma \cdot \underline{z}},\lambda_{\gamma \cdot \underline{z}},\text{d}w,-\text{d}\overline{w}\,s)|_{\underline{Y} = 0} \\\\
	&=& f(A_{\underline{z}},\iota_{\underline{z}},\psi_{\underline{z}},\lambda_{\underline{z}},\gamma^*\text{d}w,\gamma^*(-\text{d}\overline{w}\,s))|_{\underline{Y} = 0} \\\\
	&=& f\left(A_{\underline{z}},\iota_{\underline{z}},\psi_{\underline{z}},\lambda_{\underline{z}},(\text{d}w,-\text{d}\overline{w}\,s)\begin{pmatrix}
		\mu(\gamma,\underline{z})\inv & 0 \\ 0 & \mu(\gamma,\underline{z})
	\end{pmatrix}\right)|_{\underline{Y} = 0} \\\\
	&=& \mu(\gamma,\underline{z}) \cdot f\left(A_{\underline{z}},\iota_{\underline{z}},\psi_{\underline{z}},\lambda_{\underline{z}},\text{d}w,-\text{d}\overline{w}\,s\right)|_{\underline{Y} = 0}\\\\
	&=& \kappa(\mu(\gamma,\underline{z}))\phi(f)(\underline{z}).
\end{array}
\end{equation}
Thus $\phi$ is an isomorphism from the space of global sections of $\mathcal{V}_\kappa^{\nu}$ to the space of nearly holomorphic Hilbert modular forms of weight $\kappa$ and type $\nu$ in the sense of Shimura.
\end{proof}

\begin{thm}
The map $\nabla_\sigma$ is the analog of the Maass--Shimura operator, in the sense that $D_\kappa^\sigma \circ \phi = \phi \circ \nabla_\sigma$.
\end{thm}
\begin{proof}
Want to show $D_\kappa^\sigma(\phi(f))(\underline{z}) = \phi(\nabla_\sigma f)(\underline{z})$. We compute the left hand side, using the fact that $\dfrac{\partial s_\sigma}{\partial z_\sigma} = -s_\sigma^2$, while $\dfrac{\partial s_\sigma}{\partial z_\tau} = 0$ if $\sigma \neq \tau$. By the definition of the Maass--Shimura operator in Equation \eqref{definemaassshimura},
\begin{equation}
\begin{array}{rcl}
	D_\kappa^\sigma (\phi(f))(\underline{z}) &=& \kappa(s(\underline{z})) \dfrac{\partial}{\partial z_\sigma}\left[ \kappa(s(\underline{z}))\inv \phi(f)({\underline{z}})\right] \\\\
	&=& s_\sigma k_\sigma\phi(f)(\underline{z}) + \left[\dfrac{\partial}{\partial z_\sigma} \phi(f)\right](\underline{z}).
\end{array}
\end{equation}
Writing $\phi(f)(\underline{z}) = P_f(s)$ as above, we differentiate with respect to $z_\sigma$ noting that both the coefficients of $P_f$ and the input $s$ depend on $z_\sigma$.
\begin{equation}
	\frac{\partial}{\partial z_\sigma}\phi(f)(\underline{z}) = \frac{\partial}{\partial z_\sigma}P_f(s) = \frac{\partial P_f(\underline{Y})}{\partial z_\sigma}\Big|_{\underline{Y} = s} + \dfrac{\partial P_f}{\partial Y^\sigma}(s) \cdot \frac{\partial s_\sigma}{\partial z_\sigma}.
\end{equation}
We are left with
\begin{equation}
	D_\kappa^\sigma (\phi(f))(\underline{z}) = s_\sigma k_\sigma \phi(f)(\underline{z}) + \frac{\partial P_f}{\partial z_\sigma}\Big|_{\underline{Y} = s} - s_\sigma^2 \frac{\partial P_f}{\partial Y^\sigma} (s).
\end{equation}

For the right hand side, we use the fact that $\nabla_\sigma(\text{d}w,\beta) = (\text{d}w,\beta)\mu_\sigma^-$ to compute using the definition of $\nabla_\sigma$.
\begin{equation}	
\begin{array}{rcl}
	(\nabla_\sigma f)(A_{\underline{z}},\iota_{\underline{z}},\psi_{\underline{z}},\lambda_{\underline{z}},\text{d}w,\beta) &=& \left[\dfrac{\partial}{\partial z_\sigma}f + \mu_\sigma^- \cdot f\right](A_{\underline{z}},\iota_{\underline{z}},\psi_{\underline{z}},\lambda_{\underline{z}},\text{d}w,\beta) \\\\
	&=& \dfrac{\partial}{\partial z_\sigma}P_f(\underline{Y}) + \mu_\sigma^- \cdot P_f(\underline{Y}). \\\\
	&=& \dfrac{\partial P_f}{\partial z_\sigma}(\underline{Y}) + Y^\sigma\varepsilon_\sigma \cdot P(\underline{Y}) - (Y^\sigma)^2 \dfrac{\partial P_f}{\partial Y^\sigma} (\underline{Y}).
\end{array}
\end{equation}
Evaluating at $\underline{Y} = s$ gives $\phi(\nabla_\sigma f)(\underline{z})$. Since $\varepsilon_\sigma\cdot P_f=k_\sigma P_f$, we find that these agree.
\end{proof}

The fact that $\nabla_\sigma$ above is the correct analog for the Maass--Shimura operators in this formulation of nearly holomorphic Hilbert modular forms informs how we construct the analog in the overconvergent setting.

\begin{rem}\label{holoprojectionremark}
Let $f$ be a nearly holomorphic Hilbert modular form of weight $\kappa$ and type $\nu$, written as a polynomial
\begin{equation}
	f(\underline{z}) = \sum_{r=(r_\sigma)_\sigma \leq \nu} f_{r}(\underline{z})s^r = \sum_{r=(r_\sigma)_\sigma \leq \nu} f_{r}(\underline{z})\prod_\sigma s_\sigma^{r_\sigma}.
\end{equation}
The holomorphic projection of $f$ is the holomorphic function $f_0$. Implicit in this statement is the reliance on a given splitting of the Hodge filtration, which picks out $\underline{\omega}^\vee$ as a certain submodule of $\mathcal{H}$ consisting of anti-holomorphic differentials. Thus we may say that \emph{with respect to this splitting}, the holomorphic projection of $f$ is $f_0$.

Viewing nearly holomorphic Hilbert modular forms as homogeneous functions to $V_\kappa$, the holomorphic projection with respect to this splitting can be recovered by the map $V_\kappa = W_\kappa[\underline{Y}] \to W_\kappa$ which sets $\underline{Y} = 0$; i.e., realizing the form as a nearly holomorphic function and then taking its holomorphic projection is the same as setting $\underline{Y} = 0$ and then realizing the form. We note that at CM points, this splitting of the Hodge filtration is algebraic, and is still given by setting $\underline{Y} = 0$.

Finally we note that the space of splittings of the Hodge filtration is isomorphic to the coset space $H \backslash Q$: a splitting corresponds an $\mathcal{O}_F$-basis for $\mathcal{H}$ up to scale, where the first vector in the basis should generate $\underline{\omega}$ and the second generates its complement. The group $Q$ acts freely and transitively on these bases, and the action of $H$ simply rescales the vectors in the basis in such a way that it does not change the direct sum decomposition of $\mathcal{H} = \underline{\omega} \oplus \underline{\omega}^\vee$. Viewing nearly Hilbert modular forms as functions to $V_\kappa = W_\kappa \otimes \mathcal{O}_{H \backslash Q}$, we can regard the holomorphic projection with respect to a splitting of the Hodge filtration as evaluation at the point corresponding to that splitting in $\mathcal{O}_{H \backslash Q}$. Under this view, the splitting that picks out the anti-holomorphic differentials as a complement to $\underline{\omega} \subset \mathcal{H}$ serves as our ``base point" corresponding to the trivial coset $H$, with coordinates $\underline{Y} = 0$.

By viewing $V_\kappa = W_\kappa \otimes \mathcal{O}_{H \backslash Q}$ and $H \backslash Q$ as being (non-canonically) isomorphic to the space of splittings for the Hodge filtration, we get the following philosophy: a nearly automorphic form is a gadget that takes in a splitting of the Hodge filtration and spits out an honest automorphic form.
\end{rem}

\subsection{Nearly Overconvergent Hilbert Modular Forms}\label{NOHMFs}
There is a natural homomorphism $Q \to \T$ that picks out the top left entry. Let $Q_w^0$ be the rigid analytic group which is the preimage of $\T_w^0$ under this projection, and $Q_w$ the preimage of $\T_w$.

We have previously considered the $\mathcal{O}_K$-schemes $T^{\times}_{\underline{\omega}} = \operatorname{Isom}_X(\mathcal{O}_F \otimes_\Z \mathcal{O}_X,\underline{\omega})$ and $T^{\times,+}_{\mathcal{H}} = \operatorname{Isom}_X^+(\left[\mathcal{O}_F \otimes_\Z \mathcal{O}_X\right]^{\oplus 2},\mathcal{H})$, where the $+$ superscript means that the isomorphisms should respect the filtrations. We write $\mathcal{T}^\times_{\underline{\omega},an}$ and $\mathcal{T}^{\times,+}_{\mathcal{H},an}$ for their rigid analytifications, and $\mathcal{T}^\times_{\underline{\omega},an}(\underline{v})$ and $\mathcal{T}^{\times,+}_{\mathcal{H},an}(\underline{v})$ for their base changes to $\mathcal{X}(\underline{v})$. We have a map $\mathcal{T}^{\times,+}_{\mathcal{H},an}(\underline{v}) \to \mathcal{T}^\times_{\underline{\omega},an}(\underline{v})$ given by forgetting everything but the isomorphism on the submodules $\mathcal{O}_F \otimes_\Z \mathcal{O}_{\mathcal{X}(\underline{v})}$ and $\underline{\omega}$.

Then let $\mathcal{T}^\times_{\mathcal{F},w} \to \mathcal{I}_w \to \mathcal{X}(\underline{v})$ be the rigid fiber of the chain of morphisms defined in Section \ref{OverconvergentSection}. This maps to $\mathcal{T}^\times_{\underline{\omega},an}$ as well.\footnote{It does \emph{not} map to the rigid fiber of the formal completion of $T^\times_{\underline{\omega}}$, as the points of either space should be $\mathcal{O}_F \otimes_\Z \mathcal{O}_K$-bases for the respective sheaves, which do not match up. The points of $\mathcal{T}^\times_{\underline{\omega},an}$ are instead $\mathcal{O}_F \otimes K$-bases for $\underline{\omega}$.} We thus form the fiber product
\begin{equation}
	\mathcal{T}^{\times,+}_{\mathcal{H},w}(\underline{v}) \coloneqq \mathcal{T}^{\times,+}_{\mathcal{H},an}(\underline{v}) \times_{\mathcal{T}^\times_{\underline{\omega},an}(\underline{v})} \mathcal{T}^\times_{\mathcal{F},w}(\underline{v}).
\end{equation}
The map $\pi_w^+ \colon \mathcal{T}^{\times,+}_{\mathcal{H},w}(\underline{v}) \to \mathcal{X}(\underline{v})$ is a torsor for the group $Q_w$.

Let $\kappa$ be a $w$-analytic weight. There is an associated rigid analytic representation $W_{\kappa,w}$ of $\T_w$, which we inflate to $Q_w$. In similar fashion to the holomorphic case, we define the analytic $(\mathfrak{g},Q_w)$-module $V_{\kappa,w}$ by
\begin{equation}
	V_{\kappa,w} = W_{\kappa,w} \otimes \mathcal{O}_{H \backslash Q}
\end{equation}

We give $\mathfrak{g}(K)$ the same basis as before; since $Q_w$ is open in $Q$, its Lie algebra is still $\mathfrak{q} \subset \mathfrak{g}$. We define the actions of $\mathfrak{g}(K)$ and $Q_w(K)$ by the same formulas,
\begin{equation}
	(g \cdot P)(\underline{Y}) = a \cdot P(a\inv(b+\underline{Y}d)) \quad \text{for all } g = \begin{pmatrix}
		a & b \\ 0 & d
	\end{pmatrix} \in Q_w(R),
\end{equation}
\begin{equation}
	(\mu_\sigma^- \cdot P)(\underline{Y}) = Y_\sigma\varepsilon_\sigma \cdot P(\underline{Y}) - Y_\sigma^2 \frac{\partial}{\partial Y_\sigma} P(\underline{Y}).
\end{equation}

Compatibility is checked in the same way. Since the action of $Q_w$ preserves the degree of each $Y_\sigma$, there is a $Q_w$-submodule $V_{\kappa,w}^{\nu} \subset V_{\kappa,w}$ consisting of polynomials with degree at most $r_\sigma$ in the variable $Y_\sigma$ for each $\sigma$. Since $\mathfrak{g}_\sigma$ raises the degree by at most $\sigma$, $\mathfrak{g}_\sigma V_{\kappa,w}^{\nu} \subset V_{\kappa,w}^{{\nu}+\sigma}$.

\begin{defn}
The sheaf of nearly $\underline{v}$-overconvergent Hilbert modular forms of level $\Gamma_1(N)$, $w$-analytic weight $\kappa$, and type ${\nu}$ is $\mathcal{V}_{\kappa,w}^{\nu}(\underline{v}) = \mathcal{T}^{\times,+}_{\mathcal{H},w}(\underline{v}) \times_{Q_w} V_{\kappa,w}^{\nu}$. A nearly overconvergent Hilbert modular form of level $\Gamma_1(N)$, $w$-analytic weight $\kappa$, and type ${\nu}$ is a section of $\mathcal{V}_{\kappa,w}^{\nu}(\underline{v})$ for some $\underline{v}$ with each $v_i$ satisfying $0 < v_i < \frac{1}{p^w}$.
\end{defn}

Proposition \ref{ConnectionOnAssociatedBundleWithgAction} gives us a connection $\nabla_{\kappa,w} \colon \mathcal{V}_{\kappa,w} \to \mathcal{V}_{\kappa,w} \otimes \Omega^1_{\mathcal{X}(\underline{v})}(\log C)$, and the Kodaira--Spencer morphism gives us our differential operators,
\begin{equation}
	\nabla_{\sigma,w} \colon \mathcal{V}_{\kappa,w} \to \mathcal{V}_{\kappa + 2\sigma,w}
\end{equation}

The discussion of Remark \ref{RaiseTypeBySigma} applies, giving
\begin{equation}
	\nabla_{\sigma,w} \colon \mathcal{V}_{\kappa,w}^{\nu} \to \mathcal{V}_{\kappa + 2\sigma,w}^{{\nu} + \sigma}.
\end{equation}

Each $\nabla_{\sigma,w}$ commutes with $\nabla_{\tau,w}$ for any pair of embeddings $\sigma, \tau \in I$. Thus we may unambiguously write $\nabla_{\kappa^\prime,w}$ for the differential operator that raises weights by $\kappa^\prime = \sum_\sigma 2k_\sigma^\prime\sigma$, and types by $\nu^\prime = (k_\sigma^\prime)_\sigma$. We can also base change to $\mathcal{X}(\underline{v}^\prime)$ for any $\underline{v}^\prime = (v_i^\prime)$ with $0<v_i^\prime<v_i$ for each $i$.

This construction works just as well for a single weight $\kappa$ as it does for the universal $w$-analytic weight $\kappa_w^{un}$, and pulling back by the inclusion of an affinoid open $\mathcal{U} \subset \mathcal{W}_w$ allows us to apply these differential operators to families parametrized by $\mathcal{U}$.

We summarize these results in the following theorem.

\begin{thm}\label{theorem}
Fix a tuple $\underline{v} = (v_i)$ with $0<v_i<\frac{1}{p^w}$ for all $i$. For each embedding $\sigma \colon F \to K$, and any $k \geq 1$, there is a differential operator $\nabla_\sigma^k$ acting on families of nearly $\underline{v}$-overconvergent Hilbert modular forms of $w$-analytic weight, which raises the weight by $2k\sigma$ and the type by $k$. The operators $\nabla_\sigma^k$ and $\nabla_\tau^\ell$ commute for any pair of embeddings $\sigma$ and $\tau$.
\end{thm}

\subsection{Integrality}\label{IntegralitySection}
As constructed, this differential operator is defined over $K$. However, when $p$ is unramified in $F$, it is defined over $\mathcal{O}_K$. Each module ${V}_\kappa^r$ is defined over $\Z$, but our description and the basis we chose are only defined over $K$ a priori. It is true that the connection $\nabla$ is defined over $\mathcal{O}_K$ in general; however, there is a problem to defining $\nabla_\sigma$ when $p$ is ramified in $F$.

Over $K$, we can decompose $\mathfrak{g} = \bigoplus_\sigma \mathfrak{g}_\sigma$. Any trivialization $\alpha \in T^{\times,+}_{\mathcal{H}}$ over an affine open $U = \operatorname{Spec}R$ defined over $K$ determines a trivialization of $\underline{\omega}$, and thus of
\begin{equation}
	\bigoplus_\sigma \underline{\omega}_\sigma^{\otimes 2} \cong \Omega^1_X(\log C).
\end{equation}
We let $D \in T_X(U)$ be the direction dual to that trivialization, and $X(D,\alpha) \in \mathfrak{g}(R)$ such that $\nabla(D)(v)(\alpha) = \left(X(D,\alpha)\cdot v\right)(\alpha)$. In fact, we get a basis $D_\sigma^\vee$ for each $\underline{\omega}_\sigma^{\otimes 2}$; we embed that sheaf as a summand of $\Omega^1_X(\log C)$ and view the set $\{D_\sigma^\vee\}$ as a basis of the latter sheaf. Let $D_\sigma \in T_X(U)$ be dual to $D_\sigma^\vee$. We have $X(D_\sigma,\alpha) \in \mathfrak{g}_\sigma(R)$, and $\sum_\sigma D_\sigma = D$ giving $\sum_\sigma X(D_\sigma,\alpha) = X(D,\alpha)$. Specifically, $X(D,\alpha) \in \mathfrak{g}(R)$ corresponds to the tuple $(X(D_\sigma,\alpha))_\sigma \in \bigoplus_\sigma \mathfrak{g}_\sigma(R)$. Under this formalism, we have $\nabla_\sigma(v)(\alpha) = \left(X(D_\sigma,\alpha) \cdot v\right)(\alpha)$.

With this in mind, we turn to integrality. When $p$ is unramified in $F$, the decomposition $\mathfrak{g} = \bigoplus_\sigma \mathfrak{g}_\sigma$ manifests over $\mathcal{O}_K$, and $\nabla_\sigma$ acts integrally. However, when $p$ ramifies, there is no guarantee that $X(D_\sigma,\alpha) \in \mathfrak{g}(\mathcal{O}_K)$ under the assumption that $X(D,\alpha) \in \mathfrak{g}(\mathcal{O}_K)$. However, we can quantify by how much this fails, and we do so using the following lemma.

\begin{lem} \label{lemmaramificationproductoflocalrings}
Let $r = (r_\sigma)_\sigma \in \prod_\sigma \mathcal{O}_K$. There exists some $\ell$ not depending on $r$ such that $p^\ell r \in \mathcal{O}_F \otimes \mathcal{O}_K \subset \prod_\sigma \mathcal{O}_K$.
\end{lem}
\begin{proof}
Pick a $\Z$-basis $\{e_1,\dots,e_d\}$ for $\mathcal{O}_F$, and order the embeddings of $F$ into $K$, $\{\sigma_1,\dots,\sigma_d\}$. The $\mathcal{O}_K$-linear map $\mathcal{O}_F \otimes \mathcal{O}_K \to \prod_\sigma \mathcal{O}_K$ is given in this basis by the matrix
\begin{equation}
	M_{\underline{\sigma}} = \begin{pmatrix}
		\sigma_1(e_1) & \dots & \sigma_d(e_1) \\ \vdots & \ddots & \vdots \\ \sigma_1(e_d) & \dots & \sigma_d(e_d)
	\end{pmatrix}
\end{equation}
It is classical that $\det M_{\underline{\sigma}}^2 = \Delta$ is the discriminant of $F/\Q$. Further, the index of the image of this $\mathcal{O}_K$-linear map is $\#\mathcal{O}_K/(\det M_{\underline{\sigma}})$. Let $\mathfrak{p}$ be the maximal ideal in $\mathcal{O}_K$ and write $(\det M_\sigma) = \mathfrak{p}^\ell$ for some $\ell$. We have $\#\mathcal{O}_K/\mathfrak{p}^\ell = p^{f\ell}$, where $f$ is the residue field degree. Thus $p^{f\ell} (r_\sigma)_\sigma \in \mathcal{O}_F \otimes \mathcal{O}_K$ whenever $(r_\sigma)_\sigma \in \prod_\sigma \mathcal{O}_K$. This exponent does not depend on the choice of $r$, so the lemma is proven.
\end{proof}

\begin{rem}
In fact, we can say more. Since $\left(\prod_\sigma \mathcal{O}_K\right)/(\mathcal{O}_F \otimes \mathcal{O}_K)$ is a $\mathcal{O}_K$-module of order $p^{f\ell}$, it is killed by $p^\ell$ \-- the worst case scenario is that the quotient is indecomposable, and thus isomorphic to the indecomposable module $\mathcal{O}_K/\mathfrak{p}^\ell$. This module has size $p^{f\ell}$, but is killed by $p^\ell$.
\end{rem}

We use this to prove more about the Lie algebra $\mathfrak{g}$.

\begin{cor}
Let $X = (X_\sigma)_\sigma \in \mathfrak{g}(\mathcal{O}_K) \subset \prod_\sigma \mathfrak{gl}_2(\mathcal{O}_K) \subset \mathfrak{g}(K)$. View $X_\sigma$ as the tuple $(0,\dots,X_\sigma,\dots,0)$ in $\prod_\sigma \mathfrak{gl}_2(\mathcal{O}_K) \subset \mathfrak{g}(K)$. Then there exists some $\ell$ independent of $X$ such that $p^\ell X_\sigma \in \mathfrak{g}(\mathcal{O}_K)$.
\end{cor}
\begin{proof}
Pick $\ell$ as in Lemma \ref{lemmaramificationproductoflocalrings}. Note that the elements of $\mathfrak{g}(\mathcal{O}_K)$ should be viewed as $2 \times 2$ matrices with entries in $\mathcal{O}_F \otimes \mathcal{O}_K$, while $X_\sigma$ is a priori a tuple of $2 \times 2$ matrices with entries in $\mathcal{O}_K$. We may instead view $X_\sigma$ as a $2 \times 2$ matrix with entries in $\prod_\sigma \mathcal{O}_K$. Scaling $X_\sigma$ by $p^\ell$ simply scales its entries, so by our choice of $\ell$ the entries of $p^\ell X_\sigma$ are in $\mathcal{O}_F \otimes \mathcal{O}_K$, and $p^\ell X_\sigma \in \mathfrak{g}(\mathcal{O}_K)$.
\end{proof}

By the corollary, we have $p^\ell X(D_\sigma,\alpha) \in \mathfrak{g}(\mathcal{O}_K)$. Thus $p^\ell X(D_\sigma,\alpha)$ acts integrally. 
\begin{equation}
	X(D_\sigma,\alpha) \cdot \left(V_\kappa^{\nu}(\mathcal{O}_K)\right)  \subset \frac{1}{p^\ell}V_\kappa^{\nu+\sigma}(\mathcal{O}_K) \subset V_\kappa^{\nu+\sigma}(K).
\end{equation}
On the level of sheaves, the above translates to
\begin{equation}
	\nabla_\sigma(\mathcal{V}_\kappa^\nu(\mathcal{O}_K)) \subset \frac{1}{p^\ell} \mathcal{V}_{\kappa+2\sigma}^{\nu+\sigma}(\mathcal{O}_K) \subset \mathcal{V}_{\kappa+2\sigma}^{\nu+\sigma}(K).
\end{equation}
This allows us to control denominators.

\section{Descent}

\subsection{Two Groups}
The phrase ``Hilbert modular form" is ambiguous in that it can refer to the space of automorphic forms on one of two groups. The group $G = \Res_{\mathcal{O}_F/\Z}\GL_2$ is one. For any commutative ring $R$, its $R$-points are $G(R) = \GL_2(\mathcal{O}_F\otimes_\Z R)$. Determinants of such matrices live in $\T = \Res_{\mathcal{O}_F/\Z}\G_m$. Recall that over $K$, $\T$ is a split torus; its diagonal subgroup is isomorphic to $\G_m$. The second group, denoted $G^*$, is defined as the fiber product $G^* = G \times_\T \G_m$. It is a subgroup of $G$, with $R$-points $G(R) = \{g \in G(R) \mid \det(g) \in 1 \otimes R^\times \subset (\mathcal{O}_F \otimes R)^\times \}$.

Each group has its advantages. The Hecke theory for $G$ is canonical, while the theory for $G^*$ is not, producing a non-commutative Hecke algebra. On the other hand, while both are the groups of interest in a PEL type moduli problem, only the moduli problem associated to $G^*$ is actually representable by a scheme. This is exactly the reason why the previous sections were only concerned with the automorphic forms on $G^*$.

Since $G^*$ is a subgroup of $G$, we may restrict the automorphic forms on $G$ to the group $G^*$, viewing them as automorphic forms on the latter group. As in \cite{AIP2}, we record the geometric condition that distinguishes these restricted forms from other automorphic forms on $G^*$.

Let $f$ be an automorphic form on $G^*$ of weight $\nu^2 \cdot \Nm_{F/\Q}^w$. It can be extended to an automorphic form on $G$ if, for all $\epsilon \in \mathcal{O}_F^{\times,+}$, we have
\begin{equation}
f(A,\iota,\psi,\epsilon\lambda,\omega) = \nu(\epsilon)f(A,\iota,\psi,\lambda,\omega). \label{descent}
\end{equation}

We will avoid using this criterion. It is included to stress the fact that whether or not an automorphic form on $G^*$ extends to $G$ can be detected geometrically. We will opt instead to use a criterion involving the symmetric space. This will be enough, as Proposition \ref{DescentOfConnection} can be stated in terms of classical Hilbert modular forms and easily transported to the space of overconvergent ones.

\subsection{Weights}
In Section \ref{starweights}, we introduced the weight space for overconvergent modular forms on $G^*$; here we continue with the weight space for $G$, including the weight of a restricted modular form.

The space of algebraic weights for $G$ is $\{\text{characters of } \T_{/K}\} \times \Z$, with a map to the algebraic weights for $G^*$ given by $\rho \colon (\nu,w) \mapsto 2\nu + w \Nm_{F/\Q}$. In particular, if $(\nu,w)$ is a weight for $G$, $\rho(\nu,w)(g) = \nu(g)^2 \cdot (\Nm_{F/\Q}(g))^w$ for any $g \in \T(\Z_p)$. If we have a modular form of weight $(\nu,w)$ on $G$, its restriction to $G^*$ will have weight $\rho(\nu,w)$.

The space of $p$-adic weights for $G$ is $\{\text{characters of } \T(\Z_p)\} \times \Z_p$. The map $\rho$ sending a $p$-adic weight for $G$ to a $p$-adic weight for $G^*$ uses the same formula as above.

\subsection{Descent of Modular Forms}
We can quantify the difference between automorphic forms on $G^*$ and automorphic forms on $G$ using the ad\`elic viewpoint. Let $K_\infty^+$ be a maximal compact subgroup of the connected component of the identity $G(\R)^+$ in $G(\R)$, and $K_\infty^*$ the same for $G^*(\R)$. Let $K_0(N)$ and $K_0^*(N)$ be the natural choices of compact open subsets of the finite ad\`eles, consisting of matrices which are upper triangular modulo $N$. The inclusion $G^* \to G$ induces a map on the symmetric spaces $G^*(\R)^+/K_\infty^*Z(G^*(\R)) \to G(\R)^+/K_\infty^+Z(G(\R))$ which ends up being a bijection. Thus automorphic forms on $G$ and automorphic forms on $G^*$ are functions on the same space. The only difference comes from the congruence subgroups

\begin{equation}
\begin{array}{c}
\Gamma_0^*(N) = G^*(\Q) \cap K_0^*(N) = \left\{ \begin{pmatrix} a & b \\ c & d \end{pmatrix} \in \SL_2(\mathcal{O}_F) \mid c \equiv 0 \pmod{N} \right\}, \\
\Gamma_0^G(N) = G(\Q) \cap K_0(N) = \left\{ \begin{pmatrix} a & b \\ c & d \end{pmatrix} \in \GL_2^+(\mathcal{O}_F) \mid c \equiv 0 \pmod{N} \right\}.
\end{array}
\end{equation}

For any $\epsilon \in \mathcal{O}_F^{\times,+}$, let $g_\epsilon = \begin{pmatrix} \epsilon & 0 \\ 0 & 1 \end{pmatrix}$. Equation \eqref{descent} is simply a translation of the fact that $g_\epsilon \in \Gamma_0^G(N)$, while $g_\epsilon \not\in \Gamma_0^*(N)$. In a strict sense, this is the only discrepancy, as $\Gamma_0^*(N)$ and these matrices $g_\epsilon$ together generate $\Gamma_0^G(N)$.

We define an action of $\mathcal{O}_F^{\times,+}$ on the space of automorphic forms on $G^*$ of weight $\rho(\nu,w)$ to be
\begin{equation} \label{OFAction}
\epsilon \cdot f = f|g_\epsilon.
\end{equation}
The discussion above implies that $f$ is a modular form for $G$ if and only if $\mathcal{O}_F^\times$ acts on $f$ via its nebentypus.

\subsection{Descent of Operators}
At the moment, the differential operators $\nabla_\sigma$ are defined as maps that send Hilbert modular forms on $G^*$ to Hilbert modular forms on $G^*$. A priori, if $f$ is a nearly overconvergent Hilbert modular form for $G$, $\nabla_\sigma f$ is only a nearly overconvergent Hilbert modular form for $G^*$. We will argue that it extends to $G$.

\begin{prop} \label{DescentOfConnection}
The action of $\mathcal{O}_F^{\times,+}$ defined in Equation \eqref{OFAction} commutes with the Gauss--Manin connection $\nabla \colon \mathcal{H} \to \mathcal{H} \otimes \Omega^1_X$.
\end{prop}

\begin{proof}
This is classical, and seen e.g. as a special case of \cite[Proposition 12.10(2)]{Shi00}.

For another proof, we note that the slash operator is defined in such a way that it commutes with the exterior derivative $d$ defined for differential forms. Since $\nabla$ arises as a differential in the spectral sequence associated to a filtered de Rham complex whose differentials are this $d$, it commutes with the slash operator as well.
\end{proof}

This leads to our second Main Theorem.

\begin{thm}\label{secondtheorem}
The differential operators $\nabla_\sigma$ constructed in Theorem \ref{theorem} preserve the space of Hilbert modular forms for $G$ inside the space of Hilbert modular forms for $G^*$.
\end{thm}
\begin{proof}
This is essentially a special case of Proposition \ref{ConnectionOnAssociatedBundleWithgAction}. The fiber of $T^{\times,+}_\mathcal{H}$ over some point $x$ corresponding to a HBAV $(A,\iota,\psi,\lambda)$ only depends on the substring $(A,\iota)$, so we may identify the fibers over any $x$ with the fibers over $g_\epsilon \cdot x$. Under this identification, the slash operator is just the natural action of $Q \ni g_\epsilon$ on $\mathcal{V}_\kappa^{\nu}$, and we proved that this action commutes with $\nabla = \bigoplus_\sigma \nabla_\sigma$ in the previously cited proposition. Thus the fact descends to each $\nabla_\sigma$.
\end{proof}

\bibliographystyle{amsalpha}
\bibliography{AycockDiffOpsVer4Bib}

\end{document}